\documentclass{article}
\usepackage[utf8]{inputenc}
\usepackage{latexsym}
\usepackage{amssymb,amsmath}
\usepackage[amsmath,amsthm,thmmarks]{ntheorem}
\usepackage{xcolor}

\theoremstyle{plain}
\newtheorem{theorem}{Theorem}[section]
\newtheorem{prop}[theorem]{Proposition}
\newtheorem{lemma}[theorem]{Lemma}
\newtheorem{cor}[theorem]{Corollary}
\theoremstyle{definition}
\newtheorem{defi}[theorem]{Definition}

\newtheorem{ex}[theorem]{Example}

{\left(\begin{smallmatrix}}
{\end{smallmatrix}\right)}

\newenvironment{smallbmatrix}
{\left[\begin{smallmatrix}}
{\end{smallmatrix}\right]}

\newcommand{\Rc}{\mathcal{R}}

\newcommand{\dom}{{\rm dom\,}}

\newcommand{\Span}{{\rm span\,}}
\newcommand{\diag}{{\rm diag\,}}

\DeclareMathOperator{\mul}{mul}
\DeclareMathOperator{\rank}{rank}

\newcommand{\C}{\ensuremath{\mathbb C}}    
\newcommand{\N}{\ensuremath{\mathbb N}}    

\newcommand{\calR}{\mathcal R}

\newcommand{\la}{\lambda}

\renewcommand{\Im}{\operatorname{Im}}
\renewcommand{\Re}{\operatorname{Re}}

\newcommand{\Rcc}{\Rc_{\rm c}}

\title{On characteristic invariants of matrix pencils and linear relations}

\author{H.\ Gernandt\thanks{Institut f\"{u}r  Mathematik, 
Sekretariat MA 4--5, Technische Universit\"{a}t Berlin, Stra\ss e des 17.\ Juni 136, D-10623 Berlin, Germany,  ({\tt gernandt@math.tu-berlin.de}).}
\and F.\ Mart\'{\i}nez Per\'{\i}a\thanks{
Centro de Matem\'atica de La Plata -- Facultad de Ciencias Exactas, Universidad Nacional de La Plata, C.C.\ 172, (1900) La Plata,  and Instituto Argentino de Matem\'{a}tica ``Alberto P. Calder\'{o}n'' (CONICET), Saavedra 15 (1083) Buenos Aires, Argentina ({\tt francisco@mate.unlp.edu.ar}).}
\and F.\ Philipp\thanks{Institute for Mathematics,  Technische Universit\"{a}t Ilmenau, Postfach 10 05 65, D-98684 Ilmenau, Germany
({\tt friedrich.philipp@tu-ilmenau.de} and
{\tt carsten.trunk@tu-ilmenau.de}).} \and C.\ Trunk$^\ddagger$}

\date{\today}

\usepackage{graphicx}
 \usepackage{xcolor}

\begin{document}

\maketitle

\begin{abstract}
The relationship between linear relations and matrix pencils is investigated. Given a linear relation, we introduce its Weyr
characteristic. If the linear relation is the range (or the kernel) representation of a given matrix pencil, we show that there is a correspondence between this characteristic and the Kronecker canonical form of the pencil. This relationship is exploited to obtain estimations on the invariant characteristics of matrix pencils under rank one perturbations.
\end{abstract}

\section{Introduction}
Canonical forms for square matrices like the Jordan normal form are a fundamental tool in matrix analysis and uniquely describe the matrix up to similarity transformations: Two matrices are similar if their Jordan normal forms coincide. This is the case if and only if 
they share the same eigenvalues and in
these eigenvalues the Segre or the Weyr characteristic coincide.
These two characteristics count the Jordan chains at an eigenvalue in two different
ways.

The Weyr characteristic (for a given eigenvalue $\lambda$) is a finite, non-increasing sequence $w$ of natural numbers, the first entry of $w$ corresponds  to the dimension of the kernel $N(A-\lambda)$ of $A-\lambda$, the second to the dimension of the quotient space $N (A-\lambda)^2/N (A-\lambda) $, etc. The
Segre characteristic is the conjugate partition of the
 Weyr characteristic; it counts
 the dimensions of the Jordan blocks corresponding to the eigenvalue $\lambda$ or, what is the same,
it counts the length of the corresponding Jordan chains
 \cite{LS09,M33,S15,S99}.
 
 Weyr  and Segre characteristics exist also for matrix pencils. 
 Recall that matrix pencils allow a normal form: the 
Kronecker canonical form, see Appendix \ref{sec:kron}. Instead of one Segre characteristic as above,
four different types of indices appear: the finite and the infinite 
elementary divisors, the row minimal indices and the column minimal indices.
These indices measure the sizes of the different blocks in the
Kronecker canonical form. 
Moreover, two matrix pencils are strictly equivalent if and only if all the four indices coincide, see \cite{G59}, or e.g. \cite{BergTren12,BergTren13}. Thus they can be considered as Segre-type
indices for matrix pencils.
As in the case of matrices, where the Segre and the Weyr characteristics are conjugate partitions,  we define here the Weyr characteristic of a matrix pencil as the conjugate partition of the
four Segre-type indices, see Section \ref{sec:examples}.
We cite 
\cite{DD07}  and  \cite{KK86} where the terms Segre and
Weyr characteristics were used for (regular) matrix pencils.

Parallel to the Weyr characteristic of matrix pencils, we introduce the Weyr characteristic of subspaces in the cartesian product
$\mathbb C^m\times \mathbb C^m$ in Section \ref{sec:weyr}. These subspaces  are usually called
linear relations in  $\mathbb C^m$ \cite{Arens,c}. Different from the matrix
situation, it is well-known that linear relations contain not 
only Jordan chains but also \emph{chains at the eigenvalue $\infty$, singular chains,} and \emph{multi-shifts}, see \cite{BSTW20,BSTW21,BTW16,SandDeSn05,SandDeSn07}. 
A recent result   states that every linear relation
in a finite dimensional space allows a direct sum 
decomposition into four parts which are related to the
above chains \cite{BSTW21}. We use this decomposition and define in Section
\ref{sec:weyr} the Weyr characteristic of a linear relation.
As in the case of matrices and matrix pencils, it turns out
that the so defined Weyr characteristic 
is a set of invariants  of a linear relation in the 
following sense (cf.\  Section \ref{sec:weyr_kernel}): Given two linear relations in $\C^m$, their Weyr
characteristics coincide if and only 
if the linear relations are strictly equivalent.
Concerning this notion, we refer to Section \ref{sec:relations}.

The main goal of this paper is to reveal 
the connection between the  Weyr characteristic
of matrix pencils and the 
Weyr characteristic of linear relations. The connection is established
by the kernel and range representations of a matrix pencil, see also \cite{BTW16,LMPPTW18}.
In order to introduce those representations, we consider a matrix pencil of
the form
\begin{equation*}\label{ArnstadtSued}
s E-F, \quad \mbox{with}\quad  E,F\in\C^{n\times m}
\end{equation*}
If $E$ is non-singular, multiplication of the pencil from the left and the right with the inverse of $E$ leads to the pencils
\begin{equation*}
    s I -E^{-1} F\qquad \mbox{and} \qquad s I-FE^{-1},
\end{equation*}
which shows that the spectrum of the original pencil coincides with those of the matrices $E^{-1}F$ and $FE^{-1}$.
%
However, in general the matrix $E$ is singular. But its graph is a linear relation, and 
the inverse of a linear relation always exists. Hence, we consider the linear relations
\begin{equation*}
E^{-1} F = \left\{(x,y)
\in\C^{m}\times \C^m : Fx=Ey\right\}
\quad \mbox{and} \quad
FE^{-1}= \left\{ (Ex,Fx) : x\in \C^m\right\},
\end{equation*}
which were already studied in~\cite{BennByer01,BB06}. 
We
call $E^{-1} F$ the {\em kernel} and $FE^{-1}$ the {\em range} representation
(see also \cite{BTW16})
as we have
$$
E^{-1} F=N[F,-E]\qquad \mbox{ and }\qquad 
FE^{-1}=R\begin{bmatrix}E\\ F\end{bmatrix}.
$$

It is natural to compare the Weyr characteristic
of a matrix pencil  and the Weyr characteristic of the corresponding
kernel and range representations. For a given matrix pencil $sE-F$ with 
nontrivial row minimal indices and/or column minimal indices
a certain peculiarity happens: depending on the  representation (either $E^{-1}F$ or $FE^{-1}$) only the 
row minimal indices and/or column minimal indices
which are large enough are visible in the linear relation. 
We provide two examples for this in Section \ref{sec:examples}.
Nevertheless, we show 
how to calculate the Weyr characteristic of the corresponding
kernel and range representations from the Kronecker canonical form of a given matrix pencil. 


Also, considering the sizes of the matrices $E$ and $F$, it is possible to recover
the Weyr characteristic of the matrix pencil $sE-F$ from the 
Weyr characteristic of its kernel or its range representation,
 see Sections~\ref{sec:weyr_kernel} and \ref{sec:weyr_range}. Hence,
one can associate 
a linear object (the kernel or the range representation) to a 
matrix pencil. This is extremely helpful, for instance, in perturbation 
theory or stability, see e.g. \cite{Gap,Wafa}.
As a consequence, new perturbation results for the Kronecker form of matrix pencils under rank one perturbations can be derived from perturbation results for the Weyr characteristic of linear relations given in \cite{LMPPTW18}. These results can then be used to describe the controllability indices of the associated descriptor systems, see \cite{ChenResp21}.
In Section~\ref{sec:pert} below we give an application to the rank one perturbation theory for matrix pencils where both (kernel and range) representations are used.

\section{Linear relations in $\mathbb C^{m}$} \label{sec:relations}

Throughout this paper elements (pairs) from $\C^{m}\times \mathbb C^m$ will be denoted by $(x,y)$, where $x,y\in \C^{m}$. A {\em linear relation} in $\mathbb C^{m}$ is a linear subspace of $\mathbb C^{m}\times \mathbb C^m$. A matrix $A\in\mathbb C^{m\times m}$ can be identified with a linear relation in $\mathbb C^{m}$ via its
graph:
\begin{align*}
\Gamma(A):=\left\{ (x,Ax):\ x\in \C^m\right\}.
\end{align*}
For the basic notions and properties of linear relations we refer to \cite{Arens,DS2}, and for the more advanced notions see \cite{BSTW20,SandDeSn05,SandDeSn07}. Here, we denote the domain and the range of a linear relation $S$ in $X$ by $\dom S$ and $R(S)$, respectively,
\begin{align*}
 \dom S &= \left\{ x\in \C^m \ :\  (x,y) \in S  \ \text{for some $y\in\C^m$} \right\} \quad \mbox{and}\\
 R(S)&=\left\{ y\in \C^m \ :\   (x,y) \in S \ \text{for some $x\in\C^m$} \right\}.
\end{align*}
Furthermore, $N(S)$ and $\mul(S)$ denote the \textit{kernel} and the
\textit{multivalued} part of $S$,
\begin{align*}
 N(S)=\left\{ x\in \C^m \;:\; (x,0) \in S \right\} \quad \mbox{and} \quad
\mul(S)=\left\{y\in \C^m \;:\; (0,y)\in S \right\}.
\end{align*}
Obviously, a linear relation $S$ is the graph of a matrix if and only if $\mul (S)=\{0\}$.

Given a linear relation $S$ and $\lambda\in\C$, one defines $\lambda S =\left\{ (x,\lambda y) \ :\  (x,y)\in S \right\}$.
For relations $S_1$ and $S_2$ in $\C^m$ the {\em operator-like sum} $S_1+S_2$ is the relation defined by
\begin{align*}
S_1+S_2 = \left\{ (x,y+z) \ : \  (x,y) \in S_1,  \  (x,z) \in S_2 \right\}.
\end{align*}
If $S_1\cap S_2=\{0\}$ then the \emph{direct sum} is given by
\begin{align*}
S_1\oplus S_2:=\{(x_1+x_2,y_1+y_2)\ : \ (x_i,y_i)\in S_i,\; i=1,2 \}.
\end{align*}
The product of the linear relations $S_1$ and $S_2$ in $\mathbb C^m$ is the linear relation in $\mathbb C^m$ defined by
\begin{equation*}
S_1S_2 = \left\{ (x, z)\ :\ (y, z) \in S_1 \mbox{ and } (x,y) \in S_2 \;\mbox{ for some } y \in \C^m\right\}
\end{equation*}
As for matrices, the sum and product of linear relations are associative operations.

The inverse $S^{-1}$ of a linear relation $S$ in $\mathbb C^{m}$ always exists, and it is the linear relation in $\mathbb C^m$ given by
\begin{equation*}
S^{-1} =\left\{ (y,x)  \ :\  (x,y)\in S \right\}.
\end{equation*}

In the following, if $M$ is a matrix in $\mathbb C^{m\times m}$ and $S$ is a linear relation in $\C^m$, we denote by $\left[\begin{smallmatrix}M&0\\0&M\end{smallmatrix}\right]\cdot S$ the action of $M$ onto $S$:
\begin{align*}
\begin{bmatrix}M&0\\0&M\end{bmatrix}\cdot S:= \{
     (Mx, My) \in \mathbb C^{m}\times\C^m : (x,y)\in S
     \}.
\end{align*}
Given two linear relations $S_1$ and $S_2$ in $\C^m$ we say that they are \emph{strictly equivalent} if there exists an invertible matrix $T\in\C^{m\times m}$ such that
\begin{align}\label{DuaLipa}
S_2=\begin{bmatrix}
T&0 \\ 0&T
\end{bmatrix}\cdot S_1.
\end{align}
If one identifies the matrices $T$ and $T^{-1}$ with their graphs, \eqref{DuaLipa}
can be written in the sense of linear relations in the following form
\begin{align*}
S_2=TS_1T^{-1}.
\end{align*}

The notions of eigenvalue, root manifolds and point spectrum also apply to linear relations. Given $\lambda\in\C$, $S-\lambda$ stands for the linear relation $S-\lambda I$:
\begin{align*}
S-\lambda = \left\{ (x,y-\lambda x) \ :\  (x,y)\in S \right\}.
\end{align*}
Then, $\lambda\in\C$ is an \emph{eigenvalue of $S$} if $N(S-\lambda)\neq\{0\}$. On the other hand, we say that $S$ has an \emph{eigenvalue at} $\infty$ if $\mul(S)\neq\{0\}$. The point spectrum of $S$ is the set $\sigma_p(S)$ consisting of the eigenvalues $\lambda\in\C\cup\{\infty\}$ of $S$.

We denote $S^0 := I$, where $I$ denotes the identity relation (i.e.\ the graph of the identity matrix in $\C^{m}$). For $k=1,2,\ldots$ the $k$-th power of $S$ is defined recursively by
\begin{equation*}
S^k := S\cdot S^{k-1}.
\end{equation*}
Thus, we have $(x_k, x_0)\in S^k$ if and only if there exist $x_1,\ldots,x_{k-1}\in \C^m$ such that
\begin{equation}\label{Minaj}
(x_k, x_{k-1}), (x_{k-1}, x_{k-2}), \ldots, (x_1, x_0)\in S.
\end{equation}

In what follows we show that any linear relation $S$ in a finite-dimensional space admits a  decomposition into four different parts. To do so, we have
to introduce some subspaces of $\C^m$ and subrelations of $S$. 
We start with the following notions:
A sequence of pairs in $S$ such as in \eqref{Minaj} is called a {\em chain} in $S$. The chain is called {\em singular}, if $x_k=x_0=0$. 
The root manifolds $\Rc_\lambda(S)$ and $\Rc_\infty(S)$ are given by
\begin{align*}
\Rc_\lambda(S):= \bigcup_{k\in\N} N (S-\lambda)^k \quad \text{and} \quad
\Rc_\infty(S):= \bigcup_{k\in\N}  \mul S^k,
\end{align*}
respectively. 
Note that $\calR_\la(S) = \calR_0(S-\la)$ and $\calR_\infty(S) = \calR_\infty(S-\la)$ for all $\la\in\C$.
If $\lambda \in \mathbb C$ and $x\in N(S-\lambda)^{l}$, then there exist vectors $x_1,\ldots,x_{l-1}\in \C^m$ with
\begin{equation}\label{JC}
(x, x_{l-1}), (x_{l-1}, x_{l-2}), \ldots, (x_1, 0)\in S-\lambda.
\end{equation}
A chain in $S-\la$ as in \eqref{JC} is called a \emph{Jordan chain of $S$ at $\lambda$}. A chain of the form 
\begin{equation*}
(0, x_{1}), (x_{1}, x_{2}), \ldots, (x_{l-1}, x_l)\in S
\end{equation*}
is called a \emph{Jordan chain of $S$ at $\infty$}.
Note that a singular chain is always a Jordan chain at both zero and $\infty$. The notion of singular chains
will be used for a finer analysis of the root manifolds. For this,
the \emph{singular chains subspace} associated to $S$ is defined by
$$
\Rc_{\rm c}(S) := \Rc_0(S)\cap\Rc_\infty(S).
$$
The following lemma is from \cite{BSTW20}.

\begin{lemma}\label{CHain}
Let $S$ be a linear relation in a vector space $X$.
Consider a singular chain of the form
\begin{equation*}
 (0,x_{s}), (x_{s},x_{s-1}), \dots ,(x_{2},x_{1}),  (x_{1},0) \in S.
\end{equation*}
Then, for every $\lambda \in\mathbb C$, the vectors defined by
\begin{equation}\label{zk}
z_j:=\sum_{i=1}^j \begin{pmatrix} s-i \\ j-i \end{pmatrix} \lambda^{j-i}x_i,
\quad j=1,\ldots, s,
\end{equation}
satisfy
\begin{equation*}
(0,z_s), (z_s,z_{s-1}), \ldots, (z_2,z_1), (z_1,0)\in S-\lambda.
\end{equation*}
\end{lemma}
This implies that, for every $\lambda \in \C \cup \{\infty\}$, 
$$
\Rcc(S)\subset \Rc_{\lambda}(S).
$$
Moreover, for $\lambda, \mu \in \C \cup \{\infty\}$ 
with $\lambda \neq \mu$, in \cite{BSTW20} it is shown that
\begin{equation}\label{Trumpeter}
\Rcc(S)= \Rc_{\lambda}(S)\cap \Rc_{\mu}(S).
\end{equation}

Obviously, if $\Rcc(S) \ne \{0\}$  then $\sigma_p(S)=\C \cup \{\infty\}$.
Otherwise, if  $\Rcc(S) = \{0\}$, then the number of different eigenvalues in $\sigma_p(S)$ is bounded by the space dimension $m$, cf.\ 
\cite[Proposition 3.2 and Theorem 4.6]{SandDeSn05}. Hence, it is natural
to call an eigenvalue $\la\in\sigma_p(S)$ {\em degenerate} if $\calR_\la(S)\subset\Rcc(S)$ and {\em proper} otherwise.
A linear relation $S$ is called \emph{completely singular} if $S\subseteq \Rc_{c}(S) \times \Rc_{c}(S)$. If $S$ is not completely singular,
\begin{align*}
S_{\rm sing}:= S\cap (\Rc_{c}(S) \times \Rc_{c}(S)),
\end{align*}
denotes the \emph{completely singular part of $S$}.

Finally, there may exist a part of $S$ which has no Jordan chains
at all, e.g.\ 
$$
S:=\text{span} \{(x_1,x_2), (x_2,x_3), \ldots , (x_{l-1},x_l)\},
$$
where $x_1 , \ldots , x_l$ are linear independent vectors in $\mathbb C^m$.
Then this linear relation has empty spectrum. A linear relation
 with empty spectrum is called a \emph{multi-shift}.
 
In the following we present a decomposition result for 
 linear relations in finite dimensional spaces, see \cite{BSTW20}. 
To do so, we need the notion of a reducing sum decomposition, cf.\ \cite{BSTW21}.
\begin{defi}\label{decomm} \rm
Let $S$ be a linear relation in a vector space $X$. Assume that
\begin{equation*}
 X=D_1 \oplus D_2 \oplus \cdots \oplus D_n
\end{equation*}
and
\begin{equation}\label{red1}
 S=S_1 \oplus S_2 \oplus \cdots \oplus S_n,
\end{equation}
where $S_j$ is a linear relation in $D_j$ for $j=1,\ldots,n$.
Then the decomposition \eqref{red1} is called a \emph{reducing sum decomposition of~$S$ with
respect to $(D_1,\ldots,D_n)$}.
\end{defi}

The following result is taken  from~\cite[Lemma~8.1]{SandDeSn07}.
\begin{lemma}\label{Lem:PowersDecomp}
Let $S$ be a linear relation in $\mathbb C^m$ and assume that
$S$ has a reducing sum decomposition
with respect to $(D_1,\ldots,D_n)$ as in Definition~\ref{decomm}.
Let $k\in\mathbb N$, then
\begin{equation*}
    S^k= S_1^k \oplus S_2^k \oplus \cdots \oplus S_n^k,\quad
\end{equation*}
is a reducing sum decomposition of $S^k$
with respect to $(D_1,\ldots,D_n)$.
\end{lemma}

Let $S$ be a linear relation in $\mathbb C^m$ and assume that
$S$ has a reducing sum decomposition $S=S_1\oplus S_2$
with respect to $\C^m=D_1\oplus D_2$. First of all, let us show that
\begin{equation}\label{Natti1}
     N(S)=N(S_1)\oplus N(S_2).
 \end{equation}
Obviously, $N(S_1)+ N(S_2)\subset N(S)$. On the other hand, if $x\in N(S_1)\cap N(S_2)$ then
$(x,0) \in S_1$ and $(x,0) \in S_2$.
As $S = S_1\oplus S_2$ is a direct sum, $x=0$ follows and
the sum of the nullspaces is a direct sum.

Also, if $x\in N(S)$ then there exist $y_1 \in D_1$ and
 $y_2 \in D_2$ such that $(x,0)=(x_1,y_1) + (x_2,y_2)$
 with $(x_1,y_1)\in S_1$ and $(x_2,y_2)\in S_2$. As $ S_1\oplus S_2$ is a reduced sum decomposition, $y_1=y_2=0$ and the proof is completed.

By Lemma \ref{Lem:PowersDecomp} it is immediate that 
\begin{equation}\label{Natti3}
      N(S^k)=N(S_1^k)\oplus N(S_2^k) \qquad \text{and} \qquad \Rc_{0}(S)=\Rc_{0}(S_1)\oplus \Rc_{0}(S_2).
 \end{equation}
Given $\la\in\C$, replacing the linear relations $S, S_1$, and $S_2$ by $S-\lambda, S_1-\lambda$, and $S_2-\lambda$, respectively,  we conclude 
 \begin{equation}\label{Natti3AA}
      N((S-\lambda)^k)=N((S_1-\lambda)^k)\oplus N((S_2-\lambda)^k)
      \mbox{ and }
       \Rc_{\lambda}(S)=\Rc_{\lambda}(S_1)\oplus \Rc_{\lambda}(S_2).
 \end{equation}
 Moreover, if $\mu\in\mathbb C$,
$\mu \neq \lambda$, by \eqref{Trumpeter} we have
\begin{equation}\label{Natti7}
\begin{aligned}
\Rcc(S) &= \Rc_{\lambda}(S)\cap \Rc_{\mu}(S)
 = \left(\Rc_{\lambda}(S_1)\oplus \Rc_{\lambda}(S_2) \right)\cap 
 \left( \Rc_{\mu}(S_1)\oplus \Rc_{\mu}(S_2)\right)\\[1ex]
 &= \left(\Rc_{\lambda}(S_1)\cap \Rc_{\mu}(S_1) \right)\oplus
 \left( \Rc_{\lambda}(S_2)\cap \Rc_{\mu}(S_2)\right)=
 \Rcc(S_1)\oplus \Rcc(S_2).
\end{aligned}
\end{equation}
In a similar way, replacing $S,S_1$, and $S_2$ by their inverses we obtain
\begin{equation}\label{Natti5}
     \mul(S^k)=\mul(S_1^k)\oplus \mul(S_2^k).
 \end{equation}
Finally, notice that it is possible to show that
\begin{equation}\label{Natti8}
R(S^k)=R(S_1^k)\oplus R(S_2^k).
\end{equation}
If $S$ has a reducing sum decomposition as in \eqref{red1} then decompositions similar to those presented in \eqref{Natti1}-\eqref{Natti8} follow by induction. 

\medskip

Given a proper eigenvalue $\la\in\C$ of $S$, consider the subrelation 
$$S\cap \left(\Rc_{\lambda}(S)\times\Rc_{\lambda}(S)\right).$$ 
Since $S_{\rm sing}$ is contained in it and $\Rc_c(S)\subset \Rc_\lambda(S)$, it is possible to find a subspace $D_\la$ such that $\Rc_\lambda(S)=\Rcc(S)\oplus D_{\lambda}$ and a linear operator $J_\la:D_\la\rightarrow D_\la$ with $\sigma_p(J_\la)=\{\lambda\}$ such that
\begin{equation}\label{Drew}
S\cap \left(\Rc_{\lambda}(S)\times\Rc_{\lambda}(S)\right)
= S_{\rm sing}\oplus J_{\lambda}, 
\end{equation}
is a reducing sum decomposition with respect to $(\Rcc(S),D_{\lambda})$, see \cite{BSTW21}. If $\infty$ is a proper eigenvalue of $S$ then  $J_\infty$ stands for a linear relation with $\sigma_p(J_\infty)=\{\infty\}$, which is the inverse (of the graph) of a linear operator on a subspace $D_\infty$ such that $\Rc_\infty(S)=\Rcc(S)\oplus D_{\infty}$ and
\begin{equation}\label{Drews}
S\cap \left(\Rc_{\infty}(S)\times\Rc_{\infty}(S)\right)
= S_{\rm sing}\oplus J_\infty
\end{equation}
is a reducing sum decomposition with respect to $(\Rcc(S),D_{\infty})$.

\begin{theorem}\label{splitit}
Let $S$ be a linear relation in $\mathbb C^m$ with proper point spectrum $\{\lambda_1, \ldots, \lambda_l\}\subset\C\cup\{\infty\}$. Then,
there exist subspaces $D_{\la_1},\ldots,D_{\la_l}, D_M$ such that $\C^m=\Rc_c(S)\oplus D_{\la_1}\oplus \ldots \oplus D_{\la_l}\oplus D_M$ and
linear relations $S_{\rm sing}, J_{\la_1}, \ldots, J_{\la_l}, S_M$ as in \eqref{Drew} and \eqref{Drews} such that
\begin{equation}\label{DemiLovato}
S=S_{\rm sing} \oplus  J_{\la_1}\oplus \cdots\oplus J_{\la_l} \oplus S_M
\end{equation}
is a reduced sum decomposition of~$S$ with respect to
$$
\left(\Rc_{c}(S), D_{\lambda_1}, \ldots,  D_{\lambda_l},
D_M\right).
$$
The linear relation $S_{\rm sing}=S
\cap \left(\Rcc(S)\times\Rcc(S)\right)$ is completely singular and $S_M$ is a multishift. For $1\leq j\leq l$ the linear relation
$J_{\la_j}$ satisfies $\sigma_p(J_j)=\lambda_j$  and
\begin{equation}\label{Juergen}
S\cap \left(\Rc_{\lambda_j}(S)\times\Rc_{\lambda_j}(S)\right)
= S_{\rm sing}\oplus J_{\lambda_j} 
\end{equation}
is a reducing sum decomposition with respect to $(\Rcc(S),D_{\lambda_j})$.
\end{theorem}

\section{Weyr characteristic of a linear relation in $\C^m$}\label{sec:weyr}
As it was mentioned in the introduction, the Weyr characteristic of a linear operator $S$ acting on $\C^m$ is a complete set of invariants which determines the Jordan normal form of $S$. The aim of this section is to introduce a complete set of invariants for a linear relation and to show that this is an analogue of the Weyr characteristic of a linear operator.

Let $S$ be a linear relation in $\mathbb C^{m}$ and assume that it is decomposed as in \eqref{DemiLovato}.
We introduce the Weyr characteristic for a linear relation according to four different parts of it.
We distinguish between (a) the singular part of $S$, (b) each finite proper eigenvalue, and (c) the possible proper eigenvalue $\infty$. The fourth kind corresponds to the part of $S$ which has no eigenvalues.

\begin{defi}\label{CosteraAzul} \rm
Given a linear relation $S$ in $\mathbb C^{m}$, let $\lambda_1, \ldots, \lambda_l$ be its different proper eigenvalues and assume that \eqref{DemiLovato} holds.
\begin{enumerate}
\item The \emph{Weyr characteristic corresponding to a proper eigenvalue} $\lambda_j\in\C$ of $S$ is given by the sequence $W(\lambda_j):= (W_k(\lambda_j))_{k\geq 1}$, where
\begin{align*}
W_k(\lambda_j):=\dim\frac{N((S-\lambda_j)^k)+\Rcc(S)}{N((S-\lambda_j)^{k-1})+\Rcc(S)}.
\end{align*}
\item If $\infty$ is a proper eigenvalue of $S$ then the \emph{Weyr characteristic corresponding to} $\infty$ is given by the sequence $A:=(A_k)_{k\geq 1}$, where 
\begin{align*}
A_k:=\dim\frac{\mul(S^k) +\Rcc(S)}{\mul(S^{k-1})+\Rcc(S)}.
\end{align*}
\item If $S_{\rm sing}\neq \{0\}$ then the \emph{Weyr characteristic corresponding to the singular chains subspace} is defined by the sequence $B:=(B_k)_{k\geq 1}$, where
\begin{align*}
B_k:=\dim\frac{N (S^{k})\cap \Rcc(S)}{N( S^{k-1})\cap \Rcc(S)}.
\end{align*}
\item If $S_M\neq \{0\}$
then the \emph{Weyr characteristic corresponding to the multi-shifts} is given by the sequence $C:=(C_k)_{k\geq 1}$, where
\begin{align*}
C_1:=\dim\frac{R(S)+\dom S}{R(S)+\Rc_0(S)} \quad \text{and} \quad C_k:=\dim\frac{R(S^{k-1})+\Rc_{0}(S)}{R(S^{k})+\Rc_{0}(S)}, \ k\geq 2.
\end{align*}
\end{enumerate}
\end{defi}

Each of the sequences $W(\la_j)$, $A$, $B$, and $C$ is 
non-increasing \cite{BSTW21}. Since the underlying space is finite-dimensional each of the Weyr characteristic sequences contains only a finite number of non-zero terms.
In what follows, we shall treat them as finite sequences and omit the zero entries. 
Also, we collect the Weyr characteristics corresponding to the different finite proper eigenvalues in a single sequence: if these are $\lambda_1, \ldots, \lambda_l$, we set
\begin{equation*}
W:=(W(\lambda_1), W(\lambda_2), \ldots, W(\lambda_{l})).
\end{equation*}
\begin{defi}
The collection of finite sequences
\begin{equation}\label{Anitta}
\big(W,A,B,C\big)
\end{equation}
is called the \emph{Weyr characteristic} of the linear relation $S$.
If it necessary to refer to the underlying linear relation we also
write $\big(W(S),A(S),B(S),C(S)\big)$. Moreover, in the same manner, we
write then $W(S)=(W_S(\lambda_1), W_S(\lambda_2), \ldots, W_S(\lambda_{l}))$ and
$W_S(\lambda_j)= (W_k(S,\lambda_j))_{k\geq 1}$, $A(S)=(A_k(S))_{k\geq 1}$, $B(S)=(B_k(S))_{k\geq 1}$, and
 $C(S)=(C_k(S))_{k\geq 1}$.
\end{defi}

Next, we explain the quantities measured by $W(\la_j)$, $A$, $B$, and $C$. Let us start with $B$.
Since $S$ is a linear relation in a finite dimensional space, there exists $k_0\in\N$ such that $\Rc_0(S)= N(S^{k_0})$. If $[x]\neq 0$ belongs to $\frac{N(S^{k_0})\cap \Rcc(S)}{N( S^{k_0-1})\cap \Rcc(S)}$, then there exist $y_1,\ldots,y_{k_0-1}\in\C^m$ such that
\begin{align*}
(x, y_1), (y_1, y_2),\ldots, (y_{k_0-1},0)\in S.
\end{align*}
We have $x\in  \Rcc(S)$ and the chain $(x, y_1), (y_1, y_2),\ldots, (y_{k_0-1},0)$ is a Jordan chain (at 0) of length $k_0$. Moreover, the vectors $x,y_1,\ldots,y_{k_0-1}$ are linearly independent. Hence, $B_{k_0}$ is the number of linearly independent Jordan chains (at 0) of length $k_0$ in $S$ that can be extended to a singular chain. Following this reasoning, it is easy to see that for each $k$, $1\leq k \leq k_0$,
$B_{k}$ represents the number of linearly independent Jordan chains (at 0) of length at least $k$ that can be extended to a singular chain.

\smallskip

From the definition of $S_{\rm sing}$ we have that $(S^{-1})_{\rm sing}=(S_{\rm sing})^{-1}$, which implies the following lemma.
\begin{lemma}\label{Bs}
Let $S$ be a linear relation in $\C^m$. Then, for each $k\geq 1$,
$$
B_k(S)=B_k(S^{-1}).
$$ 
\end{lemma}

%

To continue with the meaning of $W(\lambda_j)$ for a complex eigenvalue $\lambda_j$ of $S$ we show that, for each $k\in\N$, the sum $W_k(\lambda_j)+B_k$ equals the dimension of $\frac{N((S-\lambda_j)^k)}{N((S-\lambda_j)^{k-1})}$. Actually, this is true not only for $\lambda_j$ but for any complex number $\lambda$.

\begin{lemma}\label{numbers}
Let $S$ be a linear relation in $\C^m$. For $k\geq 1$ and  $\lambda \in \mathbb C$, 
\begin{equation*}
W_k(S,\lambda) + B_k(S) =
\dim\frac{N((S-\lambda)^k)}{N((S-\lambda)^{k-1})}.
\end{equation*}
\end{lemma}
\begin{proof}
First, assume that $x\in\C^m$ is such that $[x]\ne 0$ in $\frac{N(S^k)\cap\Rcc(S)}{N(S^{k-1})\cap\Rcc(S)}$. Then there exist $x_1,\ldots,x_{k-1}\in\C^m$ such that
\begin{align*}
(x,x_{k-1}),(x_{k_1},x_{k-2}),\ldots, (x_1,0)\in S.
\end{align*}
Also, since $x\in \Rcc(S)$, there exist $n\in \N$ and $x_{k+1},\ldots,x_{k + n}\in \C^m$ such that
\begin{align*}
(0, x_{k+n}),(x_{k+n},x_{k+n-1}),\ldots, (x_{k+1},x)\in S.
\end{align*}
Applying Lemma \ref{CHain} to the chain $(0, x_{k+n})$, $(x_{k+n},x_{k+n-1})$, \ldots, $(x_{1},0)$ in $S$, we obtain a chain $(0,z_{k+n})$, $(z_{k+n},z_{k+n-1})$, \ldots, $(z_1,0)$ in $S-\lambda$, where each $z_j$ is  given by \eqref{zk}. In particular, $z_k$ is a linear combination of $x$ and the vectors $x_1,\ldots,x_{k-1}$. Hence, $[z_k]\ne 0$ in $\frac{N((S-\lambda)^k)\cap\Rcc(S)}{N((S-\lambda)^{k-1})\cap\Rcc(S)}$.

A similar argument shows that if $[x^1],\ldots, [x^l]$ are linearly independent in $\frac{N(S^k)\cap\Rcc(S)}{N(S^{k-1})\cap\Rcc(S)}$, then there exist linearly independent classes $[z_k^1],\ldots,[z_k^l]$ in $\frac{N((S-\lambda)^k)\cap\Rcc(S)}{N((S-\lambda)^{k-1})\cap\Rcc(S)}$. Thus, Lemma \ref{CHain} implies that
\begin{equation*}
\dim\frac{N (S^{k})\cap \Rcc(S)}{N( S^{k-1})\cap \Rcc(S)}=
\dim\frac{N ((S-\lambda)^{k})\cap \Rcc(S)}{N( (S-\lambda)^{k-1})\cap \Rcc(S)}.
\end{equation*}
Therefore, the statement in the Lemma \ref{numbers} is of the form
\begin{equation}\label{ThankPedro2}
 \dim \frac{X+Z}{Y+ Z}+\dim \frac{X\cap Z}{Y\cap Z}  =\dim \frac{X}{Y},
\end{equation}
where $X:= N((S-\lambda)^k)$, $Y:=N((S-\lambda)^{k-1})$, and $Z:=\Rcc(S)$. But for arbitrary finite-dimensional subspaces $X$, $Y$, $Z$ such that $Y\subseteq X$ it holds that 
\begin{align*}
 \dim \frac{X+Z}{Y+ Z}&=\dim(X+Z)-\dim(Y+Z)\\
 &=\dim X-\dim(X\cap Z)-\dim Y+\dim(Y\cap Z),\\
 \dim \frac{X\cap Z}{Y\cap Z}&=\dim(X\cap Z)-\dim(Y\cap Z),
\end{align*}
and summing up the above equations \eqref{ThankPedro2} follows.
\end{proof}

As a consequence of Lemma \ref{numbers}, for each $k\geq 1$, $W_k(\lambda_j)$ represents the number of linearly independent Jordan chains of $S$ at $\lambda_j$ of length at least $k$ that {\em cannot} be extended into a singular chain.

Since $\infty$ is an eigenvalue of $S$ if and only if $0$ is an eigenvalue of the inverse relation $S^{-1}$, and $\mul(S)=N(S^{-1})$, the above arguments also show that, for each $k\geq 1$, $A_k$ represents the number of linearly independent Jordan chains of $S$ at $\infty$ of length at least $k$ that cannot be extended into a singular chain. To make this precise, we formulate another lemma.

\begin{lemma}\label{numbers with A}
Let $S$ be a linear relation in $\C^m$. Then, for each $k\geq 1$, 
\begin{equation*}
A_k(S) + B_k(S) = \dim\frac{\mul(S^k)}{\mul(S^{k-1})}.
\end{equation*}
\end{lemma}
\begin{proof}
This is an immediate consequence of Lemmas \ref{Bs} and \ref{numbers}. Indeed, for every $k\geq 1$, since $A_k(S)=W_k(S^{-1},0)$ and $B_k(S)=B_k(S^{-1})$ we obtain that
\begin{align*}
A_k(S) + B_k(S) &= W_k(S^{-1},0) + B_k(S^{-1})=\dim \frac{N((S^{-1})^k)}{N((S^{-1})^{k-1})}\\
&= \dim\frac{\mul(S^k) }{\mul(S^{k-1}) }.
\end{align*}
\end{proof}

The meaning of $C$ is clearly explained in \cite{BSTW21}. Given $k\geq 1$, $C_k$ stands for the number of linearly independent chains of the form
$$
(x_1,x_2), (x_2,x_3),\ldots, (x_{k},x_{k+1})
$$ 
that are contained in $S_M$ and are composed by linearly independent vectors $x_1,\ldots,x_{k+1}\in\C^m$. We say that such a chain is a \emph{multi-shift chain} (\emph{of length $k$}) in $S_M$.
Therefore, $C_k$ represents the number of linearly independent multi-shift chains of length $k$ in $S$ that are not part of a Jordan chain of $S$ at $\la=0$. 

%

\medskip

The above considerations are made somehow more precise in Proposition \ref{DojaCat} below. 
Before, we need to 
prove a somehow technical but very helpful lemma. Given two multi-indices $a = (a_1,\ldots,a_k)\in\N^k$ and $b = (b_1,\ldots,b_l)\in\N^l$ with $k\geq l$, we define $a + b$ as the multi-index in $\N^k$ given by
$$
a + b := (a_1 + b_1, \ldots, a_l + b_l, b_{l+1}, \ldots, b_k).
$$
Note that if $a$ and $b$ are non-increasing multi-indices, then so is $a+b$.
%

\begin{lemma}\label{JeyLoh}
Let $S$ be a linear relation in $\C^m$ and assume that it has a reduced sum decomposition $S = S_1\oplus S_2$ with respect to $\C^m=D_1\oplus D_2$. Then, 
$$
A(S) = A(S_1) + A(S_2),\quad B(S) = B(S_1) + B(S_2),\quad\text{and\;\;} C(S) = C(S_1) + C(S_2).
$$
If $\la$ is a proper eigenvalue of $S_1$ or $S_2$, we have 
\begin{equation}\label{TrasBotellasalpha}
W_S(\la) = W_{S_1}(\la) + W_{S_2}(\la).
 \end{equation}
\end{lemma}
\begin{proof}

Let $\lambda$ be a proper eigenvalue of $S_1$ or $S_2$.
 Then, using \eqref{Natti3AA} and  \eqref{Natti7} we obtain
 \begin{align*}
 \ \ &\frac{N((S-\lambda_j)^k)+\Rcc(S)}{N((S-\lambda_j)^{k-1})+\Rcc(S)} \\[1ex]
  &=\frac{\left(N((S_1-\lambda_j)^k)+\Rcc(S_1)\right)
  \oplus \left(N((S_2-\lambda_j)^k)+\Rcc(S_2)\right)}{\left(N((S_1-\lambda_j)^{k-1})+\Rcc(S_1)\right)
  \oplus \left(N((S_2-\lambda_j)^{k-1})+\Rcc(S_2)\right)}
\end{align*}
which shows \eqref{TrasBotellasalpha}. The statement on $A(S)$ follows in the same way, using \eqref{Natti3AA} and \eqref{Natti5}. The same applies to $C(S)$, where one uses 
\eqref{Natti8} and \eqref{Natti3}. 

It remains to show the statement on $B(S)$. Using the same argument as above, we have 
\begin{align}\label{cafe}
\dim\frac{N((S-\lambda)^k)}{N((S-\lambda)^{k-1})} =\dim\frac{N((S_1-\lambda)^k)}{N((S_1-\lambda)^{k-1})} +
\dim \frac{N((S_2-\lambda)^k)}{N((S_2-\lambda)^{k-1})}. 
\end{align}
Then, applying Lemma \ref{numbers} repeatedly, together with \eqref{cafe}, we have 
\begin{align*}
 W_k(S,\lambda) + B_k(S)  &=\dim\frac{N((S-\lambda)^k)}{N((S-\lambda)^{k-1})}\\ &= \left(W_k(S_1,\lambda) + B_k(S_1)\right) + \left( W_k(S_2,\lambda) + B_k(S_2) \right)\\[1ex]
  &= W_k(S,\lambda) + B_k(S_1) + B_k(S_2).
\end{align*}
Therefore, $B_k(S) = B_k(S_1) + B_k(S_2)$ for each $k\geq 1$.
\end{proof}

\begin{prop}\label{DojaCat}
Given a linear relation $S$ in $\mathbb C^{m}$ assume that the reduced sum decomposition \eqref{DemiLovato} holds.
Then, the Weyr characteristic of $S$ satisfies:
\begin{align*}
&W_S(\la_j) = W_{J_{\la_j}}(\la_j), \qquad  \text{for}\ 1\leq j\leq l,\\ 
A(S)&=A(J_\infty),\quad B(S)=B(S_{\rm sing}), \quad\text{and}\quad C(S)=C(S_M).
\end{align*}
\end{prop}

\begin{proof}
Fixed $1\leq i\leq l$, applying Lemma \ref{JeyLoh} to the reducing sum decomposition \eqref{DemiLovato} we get
\begin{align*}
W_S(\la_i) &= W_{S_{\rm sing}}(\la_i)+ \sum_{j=1}^l W_{J_{\la_j}}(\la_i) + W_{S_M}(\la_i),
\end{align*}
and similar descriptions hold for $A$, $B$ and $C$.

Note that $W_{S_{\rm sing}}(\la_i)$ and $W_{S_M}(\la_i)$ are trivial sequences because $S_{\rm sing}\subset \Rcc(S)\times \Rcc(S)$ and $\Rcc(S_M)=N((S_M-\la)^k)=\{0\}$ for every $k\geq 1$ and every $\la\in\C\cup\{0\}$, respectively. Therefore, $W_S(\la_i)=\sum_{j=1}^l W_{J_{\la_j}}(\la_i)$. Moreover,  since $\sigma_p(J_{\la_j})=\{\la_j\}$ it follows that
\begin{align*}
W_S(\la_i)=\sum_{j=1}^l W_{J_{\la_j}}(\la_i)=W_{\la_i}(\la_i).
\end{align*}
Let us analyze the case of the sequence $A$. Once again, $A(S_{\rm sing})$ and $A(S_M)$ are trivial because $S_{\rm sing}\subset \Rcc(S)\times \Rcc(S)$ and $\Rcc(S_M)=\mul(S_M^k)=\{0\}$ for every $k\geq 0$, respectively. Also, if $\la_j\in\C$ then $A(J_{\la_j})$ is trivial because $\Rcc(S_{J_{\la_j}})=\mul(J_{\la_j}^k)=\{0\}$ for every $k\geq 1$ and every $1\leq j\leq l$. Hence, $A(S)=A(J_\infty)$.

According to \eqref{Juergen}, $\Rcc(J_{\la_j}) \subset \Rcc(S) \cap D_{\la_j} =\{0\}$ for every $1\leq j\leq l$. Then, $B(J_{\la_j})$ is trivial for every $1\leq j\leq l$. Also, $B(S_M)$ is trivial because $N(S_M^k)=\{0\}$ for every $k\geq 1$. Thus, $B(S)=B(S_{\rm sing})$.

To complete the proof let us show that $C(S)=C(S_M)$. We have that $C(S_{\rm sing})$ is trivial because $R(S_{\rm sing}^k)\subseteq \Rcc(S_{\rm sing})\subseteq \Rc_0(S_{\rm sing})$ for every $k\geq 1$. 

It remains to prove that $C(J_{\la_j})$ is trivial for each $j$, $1\leq j\leq l$. 
Firstly, assume that $\la_j=0$ for some $j$. In this case, $J_0$ is a nilpotent operator i.e. there exist $k_0\in\N$ such that $J_0^{k_0}=0$. Then, $\Rc_0(J_0)=N(J_0^{k_0})=D_0=\dom J_0$ and we have that $R(J_0^k)\subseteq D_0=\Rc_0(J_0)$ for each $k\geq 1$, so $C(J_0)$ is trivial. 

Secondly, assume that $\la_j=\infty$ for some $j$. Since $J_\infty$ is the inverse of the graph of a linear operator $T_0$ defined on $D_\infty$, it is immediate that $\dom J_\infty=R(T_0)\subseteq D_\infty$, $\Rc_0(J_\infty)=\{0\}$ and $R(J_\infty^k)=\dom(T_0^k)=D_\infty$ for each $k\geq 1$. Hence, $C(J_\infty)$ is trivial.
Finally, given $1\leq j\leq l$, assume that $\la_j\in\C\setminus\{0\}$. Since $0\notin \sigma_p(J_{\la_j})$ we have that $J_{\la_j}$ is invertible. In particular, $\Rc_0(J_{\la_j})=\{0\}$ and $R(J_{\la_j}^k)=D_{\la_j}=\dom J_{\la_j}$ for each $k\geq 1$, which shows that $C(J_{\la_j})$ is the trivial sequence.  
Therefore, $C(S)=C(S_M)$. 
%
%
%
 %
%
\end{proof}

 We close this section with showing that the Weyr characteristic is invariant under strict equivalence.

\begin{lemma}
\label{lem:invariance}
The Weyr characteristics of two strictly equivalent linear relations in $\C^m$ coincide.
\end{lemma}
\begin{proof}
Let $S$ and  $\tilde S$ be two strictly equivalent linear relations in $\C^m$. By definition, there
is an invertible matrix $T$ such that $(x,y)\in S$ if and only if $(Tx,Ty)\in \tilde S$. Therefore, for each $k\in \mathbb N$ and any $\lambda\in\C$,
\begin{align*}
\mul(\tilde S^k)=T(\mul(S^{k})),\quad  N((\tilde S-\lambda)^k)&=T(N((S-\lambda)^k)),\quad \Rcc(\tilde S)=T(\Rcc(S)), \\
R(\tilde S^k)=T(R(S^{k})),\quad &\text{and}\quad \Rc_r(\tilde S)=T(\Rc_r(S)).
\end{align*}
Hence,
\begin{align*}
\frac{N(\tilde S-\lambda)^k)+\Rcc(\tilde S)}{N((\tilde S-\lambda)^{k-1})+\Rcc(\tilde S)}=\frac{T\big(N((S-\lambda)^k)+\Rcc(S)\big)}{T\big(N( (S-\lambda)^{k-1})+\Rcc(S)\big)}.
\end{align*}
The invertibility of $T$ thus implies that $W(S) = W(\tilde S)$. The remaining equalities $A(S) = A(\tilde S)$, $B(S) = B(\tilde S)$, and $C(S) = C(\tilde S)$ are proved similarly.
\end{proof}

\section{Kernel and range representation for matrix pencils}
\label{sec:examples}

Kronecker proved that any pair of matrices $E,F\in\C^{n\times m}$, or what is the same
any matrix pencil of the form $sE-F$,
can be transformed into a canonical form, cf.\ \eqref{KCF} below. Nowadays, this form is referred to as the \emph{Kronecker canonical form},
see e.g. \cite{BergTren12, BergTren13, G59}. 

According to Kronecker \cite{K90}, there exist invertible matrices $U\in\C^{m\times m}$ and $V\in\C^{n\times n}$ such that $V(sE-F)U$ has the following block quasi-diagonal form
\begin{align}\label{KCF}
V(sE-F)U=\begin{bmatrix} sI_{n_0}-J_0& 0&0&0 \\0& sN_\alpha-I_{|\alpha|}&0&0\\ 0&0& sK_{\beta}-L_{\beta} &0\\ 0&0&0& sK_{\gamma}^\top-L_{\gamma}^\top\end{bmatrix}
\end{align}
where $J_0\in\C^{n_0\times n_0}$ is in Jordan normal form, and it is unique up to a permutation of its Jordan blocks, and $\alpha\in\N^{n_\alpha}$, $\beta\in\N^{n_\beta}$, $\gamma\in\N^{n_{\gamma}}$ are multi-indices which are unique up to a permutation of their entries and  the absolute value of $\alpha$ is given by $|\alpha|=\sum_{i=1}^{n_\alpha} \alpha_i$. 
For further details on the remaining blocks appearing in \eqref{KCF} we refer to Appendix \ref{sec:kron}, which is mainly based on \cite[Chapter XII]{G59}. 

For our purposes it is important to assume that the entries of the multi-indices $\alpha$, $\beta$ and $\gamma$ as well as the sizes of the different Jordan blocks in $J_0$ corresponding to a given eigenvalue $\la$ are arranged in non-increasing order.

Recall that the Weyr characteristic of a matrix arises from its Segre characteristic by computing the conjugate partition. Similar to the Segre characteristic of the matrix $J_0$, also the multi-indices $\alpha, \beta$, and $\gamma$ reflect the various block sizes in the Kronecker canonical form. Hence, it is natural to define the Weyr characteristic of a matrix pencil by conjugate partitions of the Kronecker invariants.

\begin{defi}\label{Costera}
Let $P(s)=sE-F$ be a matrix pencil with matrices 
$E$ and $F$ in $\mathbb C^{n\times m}$.
If $\lambda_1, \ldots, \lambda_{l}$ are the different eigenvalues of the matrix $J_0$ in \eqref{KCF}, 
denote by $w_{\la_1}, \ldots, w_{\la_l}$ the corresponding Weyr characteristics. Assuming that $\lambda_1, \ldots, \lambda_{l}$ are given in a certain order, we define the Weyr characteristic of $J_0$ by $w=(w_{\la_1},\ldots, w_{\la_l})$.
Also, consider the conjugate partitions of the multi-indices $\alpha$, $\beta$, and
$\gamma$ and denote them by $a=(a_1, \ldots, a_{n_{a}})$,
$b=(b_1, \ldots, b_{n_{b}})$, and
$c=(c_1, \ldots, c_{n_{c}})$, respectively.
Then, the collection of multi-indices
\begin{equation}\label{Eminem}
w(P) := (w,a,b,c)
\end{equation}
is called the \emph{Weyr characteristic of the matrix pencil} $P(s)=sE-F$.
\end{defi}

As it was explained in the introduction, a given matrix pencil $P(s)=sE-F$ gives rise to two associated linear relations, namely
\begin{align*}
E^{-1}F=N[F,-E]\qquad \text{and} \qquad FE^{-1}= R\begin{bmatrix}
E\\ F\end{bmatrix},
\end{align*}
which are the kernel and the range representation of the pencil, respectively, see \cite{BTW16}. Assuming that the Kronecker canonical form of $P(s)=sE-F$ is given by \eqref{KCF}, in \cite[Example 4.4]{BTW16} the authors showed that the eigenvalues of $P(s)$ are exactly the same as the proper eigenvalues of $FE^{-1}$. The same correspondence holds for the kernel representation $E^{-1}F$. The proof is omitted because it is similar to the arguments used for the range representation.

\begin{prop}
Given $E,F\in\C^{n\times m}$, consider the pencil $P(s)=sE-F$. If $\la\in\C\cup\{\infty\}$, the following conditions are equivalent:
\begin{itemize}
\item[\rm (a)] $\la$ is an eigenvalue of $P(s)$;
\item[\rm (b)] $\la$ is a proper eigenvalue of the linear relation $E^{-1}F$;
\item[\rm (c)] $\la$ is a proper eigenvalue of the linear relation $FE^{-1}$.
\end{itemize}
\end{prop}


In Sections~\ref{sec:weyr_kernel} and \ref{sec:weyr_range} below we describe the connection between the Weyr characteristic of the matrix pencil and of its kernel and range representation, respectively. 
 It turns out that their Weyr characteristics coincide in the  first two entries in  \eqref{Anitta} and in \eqref{Eminem}, whereas there is a certain 
shift in the last two entries. Before proving this in Sections~\ref{sec:weyr_kernel} and \ref{sec:weyr_range}, we present two simple examples.

\begin{ex} \rm
Given $E=K_{(2,1,1)}=[1,0,0,0]$ and $F=L_{(2,1,1)}=[0,1,0,0]$ in $\C^{1\times 4}$, consider the pencil 
\begin{align*}
P(s)=sE- F=[s,-1,0,0].
\end{align*}
It is given in Kronecker canonical form, and only the third block in the Kronecker canonical form is present.

According to Definition \ref{Costera}, the only non-empty entry in the Weyr characteristic of $P(s)$ is the entry $b=(b_1, \ldots, b_{n_{b}})$, which in this case is the conjugate of the multi-index $\beta=(2,1,1)$. Hence,
\begin{equation}\label{NoPuedoSalir}
b= (3,1).
\end{equation}
On the one hand, the kernel representation  of the pencil $P(s)$ is  
$$
E^{-1}F= 
\left\{\big( (x_1, x_2,x_3 ,x_4) ,  (x_2, y_1, y_2, y_3) \big)
:\ x_1, x_2, x_3, x_4, y_1, y_2, y_3\in \mathbb C \right\},
$$
and also $(E^{-1}F)^2= \mathbb C^4\times \C^4$. This implies that $\Rc_c(E^{-1}F) = \C^4$. Moreover, since $\dim N(E^{-1}F)=3$ and $\dim N(E^{-1}F)^2=4$, we obtain
$$
B(E^{-1}F)=(3,1),
$$
which coincides with \eqref{NoPuedoSalir}.

On the other hand, the range representation of the matrix pencil $P(s)$ is
$$
FE^{-1}=R \begin{bmatrix}E
\\ F\end{bmatrix} = \C\times\C.
$$
Again, this implies that $\Rc_c(FE^{-1}) = \C$ and, since $\dim N(FE^{-1})=1$,
\begin{equation}\label{MuchoTrabajo}
B(FE^{-1})=(1).
\end{equation}
Comparing \eqref{NoPuedoSalir} and \eqref{MuchoTrabajo}, we see that the range representation does not provide the complete Weyr characteristic of $P(s)$. However $B(FE^{-1})$ is part of the sequence $b$, it can be obtained by removing the first entry in $b$.
\end{ex}

\begin{ex} \rm
Given $E=K_{(2,1,1)}^\top=[1,0,0,0]^\top$ and $F=L_{(2,1,1)}^\top=[0,1,0,0]^\top$ in $\C^{4\times 1}$, consider the pencil 
\begin{align*}
Q(s)=sE- F=[s,-1,0,0]^\top.
\end{align*}
It is given in Kronecker canonical form and only the fourth block in the Kronecker canonical form is present.

The only non-zero entry in the Weyr characteristic is the entry $c=(c_1, \ldots, c_{n_{c}})$,
which is the conjugate of the multi-index $\gamma=(2,1,1)$. Hence,
\begin{equation*}
c= (3,1).
\end{equation*}
On the one hand, the kernel representation  of the pencil $Q(s)$ is
$$
E^{-1}F=\{0\},
$$
which implies that $C(E^{-1}F)$ is trivial, and $c$ cannot be recovered from $C(E^{-1}F)$.
 On the other hand, the range representation of $Q(s)$ is
 $$
 FE^{-1}= \left\{\big((x,0,0,0),(0,x,0,0) \big) :\ x\in \mathbb C \right\},
 $$
and $(FE^{-1})^2=\{0\}$. In this case, $C(FE^{-1})=(1)$, because $\Rcc(FE^{-1})=\{0\}$, $\dim R(FE^{-1})=1$ and $\dim R((FE^{-1})^2)=0$. Hence, the range representation does not provide the complete Weyr characteristic of $Q(s)$ but $C(FE^{-1})$ can be obtained from $c$ by removing the first entry in it.

\end{ex}

\section{Weyr characteristics of a pencil and its kernel representation}\label{sec:weyr_kernel}

In this section we present one of the main results of this paper. Theorem \ref{prop:KCFtoWeyr} and Proposition \ref{prop:correspondence} below show that there is a correspondence between the Weyr characteristic of a matrix pencil and that of its kernel representation. Moreover, it is shown that strict equivalence of pencils and linear relations is essentially the same as having the same Weyr characteristics. 



\begin{theorem}
\label{prop:KCFtoWeyr}
Let $P(s)=sE-F$ be a matrix pencil with matrices 
$E$ and $F$ in $\mathbb C^{n\times m}$
and Weyr characteristic $w(P) = (w,a,b,c)$. If $(W,A,B,C)$ is the Weyr characteristic of the kernel representation $E^{-1}F$ then
\begin{align*}
W=w, \quad A=a, \quad B=b.
\end{align*}
Moreover, if $c=(c_1,\ldots, c_{n_c})$, then $C=(c_3,\ldots, c_{n_c})$.
\end{theorem}

\begin{proof}
Since $w(P)$ is invariant under strict equivalence, we assume that $P(s)=sE-F$ is in Kronecker canonical form, i.e. \eqref{KCF} holds with $U=I_m$ and $V=I_n$. Then, 
 \begin{align}\label{FalklandArgentina}
 E^{-1}F= S_{n_0}\oplus S_\alpha\oplus S_\beta\oplus S_\gamma,
\end{align}
 is a reducing sum decomposition in the sense of Definition~\ref{decomm}, where $S_{n_0}=N[J_{0},-I]$, $S_\alpha= N[I_\alpha,-N_\alpha]$, $S_\beta=N[L_{\beta},-K_{\beta}]$ and $S_\gamma=N[L_{\gamma}^\top,-K_{\gamma}^\top]$.

Moreover, $S_{n_0}$ admits a reducing
sum decomposition where each subrelation is determined by one and only one Jordan block of $J_0$. Analogously, $S_\alpha$ 
admits a reducing sum decomposition where each subrelation is given by one (and only one) of the blocks along the diagonal of $N_\alpha$.
Hence, according to Lemma \ref{JeyLoh}, there is no restriction 
in assuming that the matrix $J_0$ contains only a single Jordan block $J_{n_0}(0)$
of size $n_0$ associated to the eigenvalue $\lambda=0$ and that
the multi-index $\alpha$ consists of one entry only. In this case, the Segre
characteristic of $sI-J_0$ is just the number $n_0$, whereas its Weyr characteristic is the finite sequence $(1,1,\ldots,1)\in\N^{n_0}$. Similarly, the Weyr characteristic of $sN_\alpha-I_\alpha$ is $a= (1,1,\ldots,1)\in\N^{\alpha}$.

In general, the multi-indices $\beta$ and $\gamma$ are of the form
\begin{align*}
\beta = (\beta_1, \ldots, \beta_{l_\beta}, 1,\ldots , 1), \quad
\gamma = (\gamma_1, \ldots , \gamma_{l_\gamma}, 1,\ldots , 1),
\end{align*}
where $\beta_1, \ldots, \beta_{l_\beta}$ and $\gamma_1, \ldots , \gamma_{l_\gamma}$ are integers larger than one and 
the number of entries equal to one in $\beta$ and $\gamma$ are given by $m_\beta$ and $m_\gamma$, respectively.
As with the other relations in \eqref{FalklandArgentina}, $S_{\beta}$ and 
$S_{\gamma}$ admit a reducing
sum decomposition with respect to the blocks corresponding to
$\beta_1, \ldots, \beta_{l_\beta-1}$,
$(\beta_{l_\beta}1,\ldots , 1)$, $\gamma_1, \ldots , \gamma_{l_\gamma-1}$ 
and $(\gamma_{l_\gamma}, 1,\ldots , 1)$, respectively.
Hence, in view of Lemma \ref{JeyLoh}, there is no restriction in assuming that $\beta$ is either $\beta=(\beta_1)$, $\beta=(\beta_1,1,\ldots,1)$ or $\beta=(1,\ldots,1)$ for some $m_\beta>1$. The same applies to $\gamma$. However, the computation of the Weyr characteristic $(W,A,B,C)$ of $E^{-1}F$ depends on the form of $\beta$ and $\gamma$. In the following we analyze five different cases. 

\smallskip



\emph{Case 1. }\ 
If $\beta = (\beta_1,1,\ldots , 1)\in \N^{m_\beta+1}$ for $m_\beta\neq 0$, and $\gamma= (\gamma_1,1,\ldots , 1)\in \N^{m_\gamma+1}$ for $m_\gamma\neq 0$ then the Weyr characteristic of the pencil $P(s)=sE-F$ is
 \begin{align}\label{Jungle}
 \begin{array}{ll}
w= (1,1,\ldots,1)\in\N^{n_0},& a= (1,1,\ldots,1)\in\N^{\alpha}, \\
b=(m_\beta+1,1,\ldots,1)\in\N^{\beta_1},& c=(m_\gamma + 1,1,\ldots,1)\in\N^{\gamma_1}. 
 \end{array}
\end{align}

Now, let us compute the Weyr characteristic of the kernel representation $S:=E^{-1}F$.
To do so we calculate the powers of the four linear relations in \eqref{FalklandArgentina} separately
(cf. Lemma \ref{JeyLoh}).
First, it is easy to see that $S_{n_0}$ is the graph of $J_{n_0}(0)$, and consequently $S^k_{n_0}$ is the graph of $J_{n_0}(0)^k$.
Similarly, we have that
   $$
    S_{\alpha}=  \{(x,  y) \in \mathbb C^{\alpha}\times\C^\alpha
\,:\, x=N_\alpha y\} = \{(N_\alpha y,  y)
\,:\, y\in \mathbb C^{\alpha} \},
 $$
and it is easy to see that its $k$-th power is
 \begin{equation}\label{Derulo}
 S_\alpha^k= \{(N_\alpha^k\ y,  y)\,:\, y\in \mathbb C^{\alpha} \}.
 \end{equation}
The third linear relation is of the form
\begin{align*}
\begin{split}
S_\beta &= \{(x,  y) \in \mathbb C^{\beta_1+m_\beta}\times\C^{\beta_1+m_\beta}
\,:\, L_\beta x=K_\beta y\}
\\ &=
\left\{\big((x_1, \ldots ,x_{\beta_1+m_\beta}) ,  (x_2,\ldots, x_{\beta_1}, y_1,\ldots y_{m_\beta+1}) \big) \in
\mathbb C^{\beta_1+m_\beta}\times\C^{\beta_1+m_\beta}\right\}.
\end{split}
\end{align*}
For $k<\beta_1$ the  $k$-th power is of the form
$$
S_\beta^k=\left\{\big((x_1, \ldots ,x_{\beta_1+m_\beta})) ,  (x_{k+1},\ldots, x_{\beta_1}, y_1, \ldots, y_{m_\beta +k}) \big)
\in
\mathbb C^{\beta_1+m_\beta}\times\C^{\beta_1+m_\beta}\right\},
$$
and for $k\geq \beta_1$ the first and the second components of the elements in
$S_\beta^k$ are completely unrelated to each other, that is, $S_\beta^k=\mathbb C^{\beta_1+m_\beta}\times\C^{\beta_1+m_\beta}$.
\smallskip

The last linear relation is of the form
\begin{align*}
\begin{split}
S_\gamma = \{(x,  y) \in \mathbb C^{\gamma_1-1}\times\C^{\gamma_1-1}
\,:\, L_\gamma^\top x=K_\gamma^\top y\}.
\end{split}
\end{align*}
On the one hand, if $\gamma_1=2$ then $S_\gamma = \{(0,0)\}$ and the same holds for $S_\gamma^k$.
On the other hand, if $\gamma_1 >2$ we obtain
\begin{align*}
\begin{split}
S_\gamma =
\left\{\big((x_1, \ldots ,x_{\gamma_1 -2},0),  (0, x_1,\ldots, x_{\gamma_1-2})\big)
\,:\, x_1,\ldots, x_{\gamma_1-2}  \in \mathbb C \right\}.
\end{split}
\end{align*}
For $k<\gamma_1-1$ the $k$-th power $S_\gamma^k$ is of the form
$$
\left\{\big( (x_{1}, \ldots ,x_{\gamma_1-1-k},0 , \ldots, 0),
(0 , \ldots, 0, x_{1},\ldots, x_{\gamma_1 -1-k})\big):
x_{1},\ldots, x_{\gamma_1 -1-k} \in \mathbb C \right\},
$$
and for $k\geq \gamma_1-1$ we have that $S_\gamma^k= \{(0,0)\}.$

Taking into account the decomposition \eqref{FalklandArgentina} and the above formulas for
$S_{n_0}^k, S_{\alpha}^k, S_{\beta}^k$,  and $S_{\gamma}^k$, it is easy to see that
\begin{align}\label{zwei}
\begin{split}
 N(S^k)=&  \{(x_1,\ldots, x_k, 0, \ldots, 0 ) \in \mathbb C^{n_0}\}
    \oplus \{0\}^\alpha\\[1ex]
  & \oplus \{(z_1,\ldots, z_k, 0, \ldots, 0, z_{\beta_1+1},
  \ldots , z_{\beta_1+m_\beta}) \in \mathbb C^{\beta_1+m_\beta}\}\oplus \{0\}^{\gamma_1 -1}
\end{split}
\end{align}
    and
\begin{equation}\label{drei}
    \mul(S^k)=  \{0\}^{n_0} \oplus N(N_\alpha^k) \oplus
    \{( 0, \ldots, 0, y_1,\ldots, y_{m_\beta +k})\in \mathbb C^{\beta_1+m_\beta}\}
    \oplus \{0\}^{\gamma_1 -1}.
\end{equation}
Hence,
\begin{align}\label{eins}
\begin{split}
\Rc_{0}(S)  =& \C^{n_{0}}\oplus  \{0\}^{\alpha}
\oplus\C^{\beta_1+m_\beta}\oplus\{0\}^{\gamma_1-1} \\[1ex]
\Rc_{\infty}(S)=&\{0\}^{n_0} \oplus  \mathbb C^{\alpha}
\oplus\C^{\beta_1+m_\beta}\oplus\{0\}^{\gamma_1-1}.
\end{split}
\end{align}
By definition, $\Rc_c(S)=\Rc_0(S)\cap\Rc_\infty(S)$ and we have
\begin{equation}\label{rcformula}
\Rc_c(S)=\{0\}^{n_0}\oplus \{0\}^{\alpha}\oplus \C^{\beta_1+m_\beta}\oplus \{0\}^{\gamma_1 -1}.
\end{equation}

Now, we use the \eqref{zwei}, \eqref{drei},\eqref{eins}, and
\eqref{rcformula} to compute the Weyr characteristics 
of $S$ introduced in Definition~\ref{CosteraAzul}.
The Weyr characteristic corresponding 
to the eigenvalue $\lambda=0$ has $n_0$ non zero entries.
Given $1\leq k\leq n_0$, 
\begin{align*}
W_k(0)=\dim\frac{N(S^k)+\Rc_c(S)}{N(S^{k-1})+\Rc_c(S)}
= \dim\frac{\{(x_1,\ldots, x_k, 0, \ldots, 0 ) \in \mathbb C^{n_0}\}}{\{(x_1,\ldots, x_{k-1}, 0, \ldots, 0 ) \in \mathbb C^{n_0}\}}=1,
\end{align*}
and the finite sequence $W$ is of the form $W=(1,\ldots,1)\in\C^{n_0}$, i.e. $W=w$. 
Analogously, the Weyr characteristic corresponding to $\infty$ has $\alpha$ non zero entries which are given by
\begin{align*}
A_k=\dim\frac{\mul(S^k) +\Rc_c(S)}{\mul(S^{k-1})+\Rc_c(S)} = \dim\frac{N(N_\alpha^k)}{N(N_\alpha^{k-1})}=1, \quad 1\leq k\leq\alpha.
\end{align*}
Hence, the finite sequence $A$ is of the form $A=(1,\ldots,1)\in\C^\alpha$. Therefore, $A=a$.

The Weyr characteristic of the singular chains subspace is defined by
\begin{align*}
\begin{split}
B_k=&
\dim\frac{N (S^{k})\cap \Rc_c(S)}{N(S^{k-1})\cap \Rc_c(S)}\\[1ex]
=& \left\{
\begin{array}{ll}
\dim\{(z_1, 0, \ldots, 0, z_{\beta_1+1},
  \ldots , z_{\beta_1+m_\beta}) \in 
  \mathbb C^{\beta_1+m_\beta}\}=m_\beta +1, & 
\mbox{ if } k=1,\\[1ex]
 \dim\frac{\{(z_1,\ldots, z_k, 0, \ldots, 0, z_{\beta_1+1},
  \ldots , z_{\beta_1+m_\beta}) \in 
  \mathbb C^{\beta_1+m_\beta}\}}{\{(z_1,\ldots, z_{k-1}, 0, \ldots, 0, z_{\beta_1+1},
  \ldots , z_{\beta_1+m_\beta}) \in 
  \mathbb C^{\beta_1+m_\beta}\}}=1,&
\mbox{ if } 2\leq k\leq \beta_1,
\end{array}
\right.
\end{split}
\end{align*}
and  the finite sequence $B$ coincides with $b$.

Finally, to compute $C$, note that \eqref{Derulo}
and \eqref{eins} imply 
\begin{align*}
R(S^k)+\Rc_{0}(S) = \C^{n_{0}}\oplus\C^{\alpha}
\oplus\C^{\beta_1+m_\beta}\oplus R(S_\gamma^k).
\end{align*}
If $\gamma_1 >2$ then, 
\begin{align*}
C_1=\dim\frac{R(S)+\dom S}{R(S)+\Rc_{0}(S)}=\dim\frac{\C^{n_{0}}\oplus\C^{\alpha}
\oplus\C^{\beta_1+m_\beta}\oplus \big(R(S_\gamma)+\dom(S_\gamma)\big)}{\C^{n_{0}}\oplus\C^{\alpha}
\oplus\C^{\beta_1+m_\beta}\oplus R(S_\gamma)}=1,
\end{align*}
and, for $2\leq k\leq \gamma_1-2$,
\begin{align*}
C_k=\dim\frac{R(S^k)+\Rc_{0}(S)}{R(S^{k+1})+\Rc_{0}(S)}=
\dim\frac{\{(0, \ldots, 0, x_{1},\ldots, x_{\gamma_1 -1-k}) 
\in \mathbb C^{\gamma_1-1} \}}{\{(0, \ldots, 0, x_{1},\ldots, x_{\gamma_1 -1-(k+1)}) \in \mathbb C^{\gamma_1 -1}\}}=1.
\end{align*}
Therefore, the finite sequence $C$ is
of the form $C=(1,\ldots,1)\in\N^{\gamma_1-2}$. On the other hand, if $\gamma_1\leq 2$ then $C$ is the trivial sequence.

\smallskip

\emph{Case 2. }\ Assume that $\beta=(\beta_1)$ with $\beta_1>1$ or $\gamma=(\gamma_1)$ with $\gamma_1>1$. This case can be solved exactly in the same way as Case 1, the only difference is that in this case either $m_\beta=0$ or $m_\gamma=0$.

\smallskip

\emph{Case 3. }\ Assume that $\beta= (1,\ldots , 1)\in\N^{m_\beta}$ for $m_\beta\neq 0$ and that $\gamma= (\gamma_1,1,\ldots , 1)\in\N^{m_\gamma+1}$ with $\gamma_1>1$ and $m_\gamma\geq 0$.
Also, assume that one (or both) of the first two entries of the Kronecker canonical form 
 in \eqref{KCF} is (are) present, then the  Kronecker canonical form  
of the pencil $sE-F$ has the form

 \begin{align}\label{camiloTutu}
 \begin{bmatrix} sI_{n_0}-J_0& 0& 0 &0 \\
 0& sN_\alpha-I_{\alpha}& 0_{\alpha \times m_\beta } & 0\\ 
 0&0& 0& sK_{\gamma}^\top-L_{\gamma}^\top\end{bmatrix},
 \end{align}
 where $ 0_{\alpha \times m_\beta}$ stands for the zero matrix with
$\alpha$ rows and  $m_\beta$ columns. On the contrary, if the first two entries of the Kronecker 
canonical form in \eqref{KCF} are not present, then the  Kronecker canonical form  
of the pencil $sE-F$ has the form
 \begin{align}\label{camiloTutuII}
 \begin{bmatrix}
  0_{|\gamma| \times m_\beta } & 
  sK_{\gamma}^\top-L_{\gamma}^\top\end{bmatrix}.
 \end{align}
In both cases $L_\beta$ and $K_\beta$ are zero matrices of appropriate sizes and we have that $S_\beta=\C^{m_\beta}\times \C^{m_\beta}$. Hence, the reducing sum decomposition \eqref{FalklandArgentina} reduces to $E^{-1}F=S_{n_0}\oplus S_\alpha\oplus (\C^{m_\beta}\times \C^{m_\beta})\oplus S_\gamma$, and we proceed as in Case 1 to show the statement.

\smallskip

\emph{Case 4.}\ Assume that $\beta= (\beta_1,1,\ldots , 1)\in\N^{m_\beta+1}$ with $\beta_1>1$ and $m_\beta\geq 0$ and $\gamma = (1,\ldots , 1)\in\N^{m_\gamma}$ with $m_\gamma\neq 0$.
Then, the reducing sum  decomposition \eqref{FalklandArgentina} is of the form $E^{-1}F=S_{n_0}\oplus S_{\alpha} \oplus S_{\beta}$, where
 \begin{align*}
S_\beta=   N\left[\begin{bmatrix}
 L_{\beta_1}\\
0_{m_\gamma\times |\beta_1|}
\end{bmatrix}, -\begin{bmatrix}
 K_{\beta_1}\\
0_{m_\gamma\times |\beta_1|}
\end{bmatrix}
\right] =  N\left[L_{\beta_1}, -K_{\beta_1}
\right],
\end{align*}
and the proof of the statement follows as in Case 1.

\smallskip

\emph{Case 5. }\ Let $\beta = (1,\ldots , 1)\in\N^{m_\beta}$ with $m_\beta\geq 1$ and $\gamma = (1,\ldots , 1)\in\N^{m_\gamma}$ with $m_\gamma\geq 1$. If one (or both) of the first two entries of the KCF in \eqref{KCF} is (are) present, then it is of the form
\begin{align*}
\begin{bmatrix} sI_{n_0}-J_0& 0& 0 \\
 0& sN_\alpha-I_{\alpha}& 0 \\ 
 0&0&  0_{m_\gamma \times m_\beta } 
\end{bmatrix}.
\end{align*}  
Then the reducing sum decomposition \eqref{FalklandArgentina} is of the form $E^{-1}F=S_{n_0}\oplus S_\alpha\oplus S_\beta$, where $S_\beta=\C^{m_\beta}\times\C^{m_\beta}$ because $L_\beta$ and $K_\beta$ are zero matrices of appropriate sizes. Finally, if the first two entries are not present, the Kronecker canonical form equals $0_{m_\gamma \times m_\beta}$ and the reduced sum decomposition is just $E^{-1}F=S_\beta=\C^{m_\beta}\times\C^{m_\beta}$. Therefore, in both cases $C_k=0$ for every $k\geq 1$, $B_1=m_\beta=b_1$ and the equalities $W=w$ and $A=a$ follows as in Case 1.
\end{proof}

Theorem~\ref{prop:KCFtoWeyr} shows how the Weyr characteristic $(w,a,b,c)$ of a matrix pencil determines the Weyr characteristic $(W,A,B,C)$ of its kernel representation. As is shown, the multi-index $C$ differs from $c$ by two entries, $c_1$ and $c_2$ vanish.

In order to reconstruct the Weyr characteristic of a matrix pencil from
its kernel representation, it is necessary to know not only the Weyr characteristic of the linear relation but also the sizes of the matrices in the kernel representation (these matrices are not uniquely given).

\begin{prop}
\label{prop:correspondence}
Let $S$ be a linear relation in $\mathbb C^{m}$ with Weyr characteristic $(W,A,B,C)$. If there exist matrices  $E,F\in\C^{n\times m}$ such that the kernel representation $E^{-1}F$ of the matrix pencil $P(s)=sE-F$ coincides with $S$, then the Weyr characteristic $(w,a,b,c)$ of the pencil $P(s)$ is given by
\begin{eqnarray*}
w=W, \quad a=A, \quad b=B,
\end{eqnarray*}
and, if $C=(C_1,\ldots, C_{n_C})$ is the last multi-index in the Weyr characteristic
of the  linear relation $S$ then
\begin{equation}\label{c}
c= (n-m+B_1, m-|W|-|A|-|B|-|C|, C_1, C_2, C_3,\ldots, C_{n_C}).
\end{equation}
\end{prop}

\begin{proof}
Assume that $S=E^{-1}F$ is the kernel representation of the matrix pencil $P(s)$. Then, there exist invertible matrices $U\in \C^{n\times n}$ and $V\in\C^{m\times m}$ such that $U P(s) V= sUEV- UFV$ is in Kronecker canonical form. 
Note that $(x,y)\in (UEV)^{-1}(UFV)$ if and only if $(Vx,Vy)\in E^{-1}F$.
This implies that $(UEV)^{-1}(UFV)$ and $E^{-1}F$ are strictly equivalent because
\begin{align*}
(UEV)^{-1}(UFV)=\begin{bmatrix}
V^{-1} & 0 \\ 0 & V^{-1}
\end{bmatrix}\cdot E^{-1}F.
\end{align*}
Hence, according to Lemma~\ref{lem:invariance}, they have the same Weyr characteristic. Therefore, we can assume without loss of generality that $P(s)$ is already in Kronecker canonical from.

Now, let $(w,a,b,c)$ be the Weyr characteristic of the pencil $P(s)$. 
It is possible to express the numbers
$m$ and $n$ of the matrix pencil $P(s)$ in~\eqref{KCF} in terms of the Weyr characteristic. In fact, $|w|$ and $|a|$ are the number of rows  (and columns) of $P(s)$ determined by the blocks associated to the finite and the infinity eigenvalues, respectively. On the other hand, $b_1$ and $c_1$ are the number of rectangular blocks corresponding to the pencils $sK_\beta-L_\beta$ and $sK_\gamma^\top-L_\gamma^\top$, respectively. In particular, the number of rows and columns of $P(s)$ determined by the pencil $sK_\beta-L_\beta$ are $|b|-b_1$ and $|b|$, respectively. Analogously, the number of rows and columns of $P(s)$ determined by the pencil $sK_\gamma^\top-L_\gamma^\top$ are $|c|$ and  $|c|-c_1$, respectively. Hence,
\begin{equation}\label{Norton}
n= |w|+|a|+(|b|-b_1)+|c| \quad \text{and}\quad m=|w|+|a|+|b|+(|c|-c_1).
\end{equation}


By Theorem \ref{prop:KCFtoWeyr}, we have that
\begin{equation*}
w=W, \quad a=A, \quad \text{and} \quad b=B.
\end{equation*}
Also, we know the indices in $c$ with the exception of the first two: $c_k=C_{k-2}$ if $k\geq 3$.
Moreover, $c_1\geq c_2\geq c_3=C_1$ and
\begin{align*}
|c|=|C| + c_1 + c_2.
\end{align*}
Replacing the above formula in \eqref{Norton}, we obtain
\begin{align*}
n &= |W|+ |A| + (|B|-B_1) + |C| + c_1 + c_2, \\
m &= |W| + |A| + |B| + |C| + c_2.
\end{align*}
Thus,
\begin{align*}
c_2 &=  m - |W| - |A| - |B| - |C|, \quad \text{and}\\
c_1 &=  n- |W|- |A| - (|B|-B_1) - |C|-c_2\\
&= n- |W|- |A| - (|B|-B_1) - |C|-(m - |W| - |A| - |B| - |C|)\\
&= n-m + B_1.
\end{align*}
\end{proof}

\begin{cor}\label{minrow}
Let $S$ be a linear relation in $\mathbb C^{m}$ and assume that its Weyr characteristic is $(W,A,B,C)$. Then,
\begin{align*}
n=2m-|W|-|A|-|B|-B_1-|C|,
\end{align*}
is the smallest possible number of rows for matrices $E,F\in\C^{n\times m}$ such that the kernel representation $E^{-1}F$ of the matrix pencil $P(s)=sE-F$ coincides with $S$. In this case, the Weyr characteristic $(w,a,b,c)$ of the pencil $P(s)$ is given by
\begin{eqnarray*}
w=W, \quad a=A, \quad b=B,
\end{eqnarray*}
and, if $C=(C_1,\ldots, C_{n_c})$ is the last multi-index in the Weyr characteristic
of the  linear relation $S$, then
\begin{align*}
c= (m - |W| - |A| - |B| - |C|, m - |W| - |A| - |B| - |C|, C_1, C_2, C_3,\ldots, C_{n_C}).
\end{align*}
\end{cor}

\begin{proof}
Assume that $E,F\in\C^{n\times m}$ are such that the kernel representation $E^{-1}F$ of the matrix pencil $P(s)=sE-F$ coincides with $S$. By Proposition~\ref{prop:correspondence}, the Weyr characteristic $(w,a,b,c)$ of the pencil $P(s)$ is such that $w=W$, $a=A$, $b=B$, and $c$ is given by \eqref{c}.

First of all, let us show that the proposed minimal number of rows is positive. According to \eqref{Norton}, $m=|W|+|A|+|B|+|C|+c_2$. Then,
\begin{align*}
2m-|W|-|A|-|B|-B_1-|C| &=m -B_1 + (m-|W|-|A|-|B|-|C|) \\ &=(m-B_1) + c_2\geq 1. 
\end{align*}
Also, according to \eqref{Norton} we have that $n= |w|+|a|+(|b|-b_1)+|c|$. Then, in order to minimize $n$ it is necessary to minimize $|c|$, and the minimum of $|c|$ is attained only if $c_1=c_2=c_3$.
Therefore, using the expressions of $c_1$ and $c_2$ in \eqref{c} we get $n-m+B_1=m-|W|-|A|-|B|-|C|$, i.e.
\begin{align*}
n=2m - |W|-|A|-|B|-B_1-|C|.
\end{align*}
\end{proof}


In the following we are going to prove the converse of Lemma \ref{lem:invariance}, i.e.\ if two linear relations $S_1$ and $S_2$ in $\C^m$ have the same Weyr characteristic, then they are strictly equivalent, see Theorem \ref{Weeknd} below.

\begin{prop}\label{casi}
Given $E_1,E_2,F_1,F_2\in\C^{n\times m}$, consider the matrix pencils $P_1(s)=sE_1-F_1$ and $P_2(s)=sE_2-F_2$. Let $S_1=E_1^{-1}F_1$ and $S_2=E_2^{-1}F_2$ be their corresponding kernel representations. If $P_1(s)$ is strictly equivalent to $P_2(s)$, 
then $S_1$ is strictly equivalent to $S_2$. In this case, the Weyr
characteristics of $S_1$ and $S_2$ coincide.
\end{prop}

\begin{proof}
Assume that $P_1(s)$ is strictly equivalent to $P_2(s)$, i.e.
there exist invertible matrices $U\in\C^{n\times n}$ and $V\in \C^{m\times m}$ such that
\begin{align*}
E_2=UE_1V \quad \text{and} \quad F_2=UF_1V,
\end{align*}
 (cf.\ Appendix \ref{sec:kron}). Then, $S_2=E_2^{-1}F_2=(UE_1V)^{-1}(UF_1V)$, or equivalently,
\begin{align*}
S_2 &= \left\{ (x,y): \ F_1Vx=E_1Vy \right\} \\ &=\begin{pmatrix}
V^{-1}&0 \\ 0& V^{-1}
\end{pmatrix}\cdot\left\{ (Vx,Vy): \ F_1Vx=E_1Vy \right\} \\
&= \begin{pmatrix}
T&0 \\ 0&T
\end{pmatrix}\cdot S_1,
\end{align*}
where $T:=V^{-1}$. Therefore, $S_1$ is strictly equivalent to $S_2$. The last statement about the Weyr characteristics  
follows from Lemma \ref{lem:invariance}.
\end{proof}

\begin{theorem}\label{Weeknd}
Let $S_1$ and $S_2$ be linear relations in $\C^m$. Then their Weyr
characteristics coincide if and only 
if the linear relations $S_1$ and $S_2$  are strictly equivalent.
\end{theorem}

\begin{proof}
On the one hand, in Lemma \ref{lem:invariance} we showed that if $S_1$ and $S_2$  are strictly equivalent then their Weyr
characteristics coincide.

On the other hand, if $S_1$ and $S_2$ have the same Weyr characteristic then, defining $n$ as the minimal value given in Corollary~\ref{minrow}, there exist matrices $E_1,F_1,E_2,F_2\in\C^{n\times m}$ such that
$$
S_1=E_1^{-1}F_1 \quad \text{and} \quad S_2=E_2^{-1}F_2.
$$
Moreover, the Weyr characteristics of the associated matrix pencils $P_1(s)=sE_1-F_1$ and $P_2(s)=sE_2-F_2$ coincide. Therefore, $P_1$ is strictly equivalent to $P_2$ and, by Proposition \ref{casi}, the kernel representations $S_1$ and $S_2$ are also strictly equivalent.
\end{proof}

\section{Weyr characteristics of a pencil and its range representation}\label{sec:weyr_range}

Along this section we obtain results on the correspondence between the Kronecker canonical form of a matrix pencil and  the Weyr characteristic of its range representation. They are similar to those presented in Section \ref{sec:weyr_kernel}.

\begin{theorem}
\label{prop:KCFtoWeyr2}
Let $P(s)=sE-F$ be a matrix pencil with matrices 
$E$ and $F$ in $\mathbb C^{n\times m}$ and with Weyr  characteristic $w(P) = (w,a,b,c)$.
If $(W,A,B,C)$ is the Weyr characteristic of the range representation $FE^{-1}$ then
\begin{equation*}
W=w,\quad A=a,
\end{equation*}
and, if $b=(b_1,\ldots, b_{n_b})$ and $c=(c_1,\ldots, c_{n_c})$, then
\begin{align*}
B=(b_2,\ldots, b_{n_b}) \quad \text{and} \quad C=(c_2,\ldots, c_{n_c}).
\end{align*}
\end{theorem}

The proof of Theorem \ref{prop:KCFtoWeyr2} has the same
structure as the proof of Theorem \ref{prop:KCFtoWeyr}. We concentrate our attention only on those calculations which are different.

\begin{proof}
There is no restriction to assume that $P(s)=sE-F$ is in Kronecker canonical form. As in the proof of Theorem \ref{prop:KCFtoWeyr},
we assume that the matrix $J_0$ in
 \eqref{KCF} contains only a single Jordan block $J_{n_0}(0)$ of size $n_0$ associated to the eigenvalue $\lambda=0$
and that the multi-index $\alpha$ consists of one entry only.
Once again, there is no restriction in assuming that $\beta$ is either $\beta=(\beta_1)$, $\beta=(\beta_1,1,\ldots,1)$ or $\beta=(1,\ldots,1)$ for some $m_\beta>1$, and the same applies to $\gamma$.

The Kronecker canonical form induces a reducing sum decomposition 
of the range representation $FE^{-1}$ as
 \begin{align}\label{Heino}
FE^{-1} = T_{n_0}\oplus
  T_\alpha\oplus
   T_\beta\oplus
    T_\gamma,
\end{align}
where 
$$
T_{n_0}:=R\begin{bmatrix}I\\ J_0\end{bmatrix}, \  T_\alpha:=R\begin{bmatrix}N_\alpha\\ I_\alpha\end{bmatrix}, \  T_\beta:=R\begin{bmatrix}
K_{\beta}\\ L_{\beta}\end{bmatrix}, \mbox{ and } T_\gamma:=R\begin{bmatrix} K_{\gamma}^\top\\ L_{\gamma}^\top
\end{bmatrix}.
$$

\emph{Case 1. }\ If $\beta = (\beta_1,1,\ldots , 1)\in\N^{m_\beta+1}$ with $\beta_1>1$ and $\gamma= (\gamma_1,1,\ldots , 1)\in\N^{m_\gamma+1}$ with $\gamma_1>1$ then
the Weyr characteristic of the pencil $P(s)$ is given again by \eqref{Jungle}.

As in the proof of Theorem \ref{prop:KCFtoWeyr}, we calculate the powers of the four linear 
relations in \eqref{Heino} separately (cf. Lemma \ref{JeyLoh}).
Obviously, $T_{n_0}$ is the graph of $J_0$, and consequently $T^k_{n_0}$ is the graph of $J_0^k$.
Similarly, 
\begin{equation}\label{Scooter}
    T_{\alpha}=  \{(N_\alpha y,  y)\,:\, y\in \mathbb C^{\alpha} \}
\mbox{ and } T_\alpha^k= \{(N_\alpha^k y,  y)\,:\, y\in \mathbb C^{\alpha} \}.
\end{equation}

The third linear relation is of the form 
\begin{align*}
T_\beta=
\left\{\big((y_1, x_1, \ldots ,x_{\beta_1-2}) ,  (x_1,\ldots, x_{\beta_1-2}, z_1)\big) 
: (x_1,\ldots, x_{\beta_1-2})\in \mathbb C^{\beta_1-2}, y_1, z_1\in \mathbb C \right\}.
\end{align*}
For $k\geq \beta_1-1$ we have that $T_\beta^k=\mathbb C^{\beta_1-1}\times\C^{\beta_1-1}$, and
for $k<\beta_1-1$ the pairs in the $k$-th power $T_\beta^k$ are of the form
$$
\big((y_1,\ldots, y_k, x_1, \ldots ,x_{\beta_1-k-1}) ,  (x_1,\ldots, x_{\beta_1-k-1}, z_1,
\ldots, z_k)\big),
$$
where $x_1,\ldots, x_{\beta_1-k-1},y_1, \ldots y_k, z_1,\ldots, z_k\in \mathbb C$.

The pairs in the last linear relation $T_\gamma$ are of the form
\begin{align*}
\begin{split}
\big( (x_1, \ldots ,x_{\gamma_1 -1}, 0, \ldots ,0), (0, x_1,\ldots, x_{\gamma_1-1}, 0, \ldots, 0)\big)
\in \mathbb C^{\gamma_1 + m_\gamma} \times  \mathbb C^{\gamma_1 + m_\gamma},
\end{split}
\end{align*}
where $ x_1,\ldots, x_{\gamma_1 -1}  \in \mathbb C$.
For $k\geq \gamma_1$ we have that
$T_\gamma^k= \{(0,0)\}$ and,
for $k<\gamma_1$ the pairs in the $k$-th power $T_\gamma^k$ are given by
$$
\big((x_{1}, \ldots ,x_{\gamma_1-k},0,\ldots,0), (0,\ldots,0,  x_{1},\ldots, x_{\gamma_1 -k}, 0, \ldots, 0)\big)
\in \mathbb C^{\gamma_1 + m_\gamma} \times  \mathbb C^{\gamma_1 + m_\gamma}.
$$
If $T:=FE^{-1}$ then the following descriptions hold:
\begin{align*}
\begin{split}
 N(T^k)=&  \{(x_1,\ldots, x_k, 0, \ldots, 0 ) \in \mathbb C^{n_0}\}
    \oplus \{0\}^\alpha \\[1ex]
  & \oplus \{(y_1,\ldots, y_k, 0, \ldots, 0) \in \mathbb C^{\beta_1-1}\}\oplus \{0\}^{\gamma_1 +m_\gamma}
\end{split}
\end{align*}
    and
\begin{equation*}
    \mul(T^k)=  \{0\}^{n_0} \oplus N(N_\alpha^k) \oplus
    \{(z_1,\ldots, z_{k}, 0, \ldots, 0)\in \mathbb C^{\beta_1-1}\}
    \oplus \{0\}^{\gamma_1 +m_\gamma}.
\end{equation*}
Hence,
\begin{align}\label{eins2}
\begin{split}
\Rc_{0}(T)  =& \C^{n_{0}}\oplus  \{0\}^{\alpha}
\oplus\C^{\beta_1-1}\oplus\{0\}^{\gamma_1+m_\gamma}, \\[1ex]
\Rc_{\infty}(T)=&\{0\}^{n_0} \oplus  \mathbb C^{\alpha}
\oplus\C^{\beta_1-1}\oplus\{0\}^{\gamma_1+m_\gamma}, \\[1ex]
\Rc_c(T)=&\{0\}^{n_0}\oplus \{0\}^{\alpha}\oplus \C^{\beta_1-1}\oplus \{0\}^{\gamma_1 +m_\gamma}.
\end{split}
\end{align}
Therefore, the first two parts of the Weyr characteristics coincide,
\begin{equation*}
W=w \quad \text{and} \quad A=a.
\end{equation*}
The entries of the Weyr characteristic of the singular chains subspace are
\begin{align*}
B_k=\dim\frac{N (T^{k})\cap \Rc_c(T)}{N( T^{k-1})\cap \Rc_c(T)}=
\dim\frac{\{(y_1,\ldots, y_k, 0, \ldots, 0 ) \in \mathbb C^{\beta_1-1}\}}
{\{(y_1,\ldots, y_{k-1}, 0, \ldots, 0 ) \in \mathbb C^{\beta_1-1}\}}=1,
\end{align*}
if $1\leq k\leq \beta_1-1$, and they are zero otherwise. So, $B=(1,\ldots,1)\in\N^{\beta_1-1}$.

For the Weyr characteristic of multi-shifts, note that \eqref{Scooter} and \eqref{eins2} imply
\begin{align*}
R(T^k)+\Rc_{0}(T) = \C^{n_{0}}\oplus\C^{\alpha}
\oplus\C^{\beta_1 -1}\oplus R(T_\gamma^k).
\end{align*}
Hence, the entries of $C$ are given by
\begin{align*} 
C_1 & = \dim\frac{R(T)+\dom T}{R(T)+\Rc_{0}(T)}=\dim\frac{\C^{n_{0}}\oplus\C^{\alpha}
\oplus\C^{\beta_1 -1}\oplus \big(R(T_\gamma) +\dom(T_\gamma)\big)}{\C^{n_{0}}\oplus\C^{\alpha}
\oplus\C^{\beta_1 -1}\oplus R(T_\gamma)}=1,\\
C_k & =\dim\frac{R(T^k)+\Rc_{0}(T)}{R(T^{k+1})+\Rc_{0}(T)}=
\dim\frac{\{(0, \ldots, 0, x_{1},\ldots, x_{\gamma_1 -k}) \in \mathbb C^{\gamma_1+m_\gamma} \}}{\{(0, \ldots, 0, x_{1},\ldots, x_{\gamma_1 -k-1}) \in \mathbb C^{\gamma_1+m_\gamma}\}}=1,
\end{align*}
if $2\leq k\leq \gamma_1-1$, and they are zero otherwise. Thus, $C=(1,\ldots,1)\in\N^{\gamma_1-1}$. Comparing the obtained $(W,A,B,C)$ with \eqref{Jungle} we complete the proof of this case.

\emph{Case 2. }\ Assume that $\beta=(\beta_1)$ with $\beta_1>1$ or $\gamma=(\gamma_1)$ with $\gamma_1>1$. As in the proof of Theorem \ref{prop:KCFtoWeyr}, this follows as in Case 1 with the condition $m_\beta=0$ or $m_\gamma=0$, respectively.

\emph{Case 3. }\ Assume that $\beta = (1,\ldots , 1)\in\N^{m_\beta}$ for $m_\beta>0$ and also that $\gamma= (\gamma_1,1,\ldots , 1)\in\N^{m_\gamma+1}$  with $\gamma_1>1$ and $m_\gamma\geq 0$. Then the  Kronecker canonical form  
of the pencil $sE-F$ has the form \eqref{camiloTutu} or \eqref{camiloTutuII}. In both cases
\eqref{Heino} induces a reducing sum decomposition with $T_\beta=\{0\}$. Hence, we can proceed as
 in Case 1 to prove the statement.

\emph{Case 4. }\ Assume that $\beta= (\beta_1,1,\ldots , 1)\in\N^{m_\beta+1}$ with $\beta_1>0$ and $\gamma = (1,\ldots , 1)\in\N^{m_\gamma}$ for some $m_\gamma>0$.
Then the Kronecker canonical form \eqref{KCF} 
of the pencil $sE-F$ is of the form
\eqref{Derulo4} and we obtain the following reducing sum  decomposition 
 \begin{align*}
 FE^{-1}= R\begin{bmatrix} 
 I\\J_0
 \end{bmatrix}
 \oplus 
  R\begin{bmatrix} 
 N_\alpha\\I_\alpha
 \end{bmatrix}
 \oplus
   R\begin{bmatrix} 
 K_\beta\\0_{m_\gamma\times |\beta|} \\ L_\beta\\0_{m_\gamma\times |\beta|}
 \end{bmatrix},
\end{align*}
where $0_{m_\gamma\times |\beta|}$ stands for the zero matrix with
$m_\gamma$ rows and  $|\beta|$ columns. Then the proof of this case is similar to the proof of Case 1.

\emph{Case 5. }\ Assume that $\beta =(1,\ldots , 1)\in\N^{m_\beta}$ for $m_\beta>0$ and $\gamma = (1,\ldots , 1)\in\N^{m_\gamma}$ for $m_\gamma>0$. Again, if one (or both) of the first two entries of the KCF in \eqref{KCF} is (are) present, then it is of the form 
\begin{align*}
\begin{bmatrix} sI_{n_0}-J_0& 0& 0 \\
 0& sN_\alpha-I_{\alpha}& 0 \\ 
 0&0&  0_{m_\gamma \times m_\beta } 
\end{bmatrix}.
\end{align*}  
Then the reducing sum decomposition \eqref{FalklandArgentina} is of the form $T=FE^{-1}= T_{n_0}\oplus T_\alpha \oplus T_\gamma$, where $T_\gamma=\{(0,0)\}^{m_\gamma\times m_\gamma}$ because $L_\gamma$ and $K_\gamma$ are zero matrices of appropriate sizes. Since $\Rcc(T)=\{0\}^{n_0+\alpha +m_\gamma}$, we have that $B_k=0$ for every $k\geq 1$. Also, $\Rc_0(T)+ R(T^k)=\dom T + R(T)=\C^{n_0}\oplus \C^\alpha\oplus\{0\}^{m_\gamma}$ implies that $C_k=0$ for each $k\geq 1$.

Finally, if the two first entries are not present then the KCF equals $0_{m_\gamma \times m_\beta}$ and the reduced sum decomposition is just $FE^{-1}= T_\gamma=\{(0,0)\}^{m_\gamma\times m_\gamma}$. Once again, it is easy to see that $B_k=C_k=0$ for every $k\geq 1$.
\end{proof}


We end this section with versions of Proposition \ref{prop:correspondence} and Corollary \ref{minrow} corresponding to the range representation of a matrix pencil. The proofs are omitted since they use similar arguments.

\begin{prop}
Let $S$ be a linear relation in $\mathbb C^{n}$ with Weyr characteristic $(W,A,B,C)$. If there exist matrices  $E,F\in\C^{n\times m}$ such that the range representation $FE^{-1}$ of the matrix pencil $P(s)=sE-F$ coincides with $S$, then the Weyr characteristic $(w,a,b,c)$ of the pencil $P(s)$ is given by
\begin{equation*}
w=W, \quad a=A,
\end{equation*}
and if $B=( B_1,\ldots, B_{n_B})$ and $C=( C_1,\ldots, C_{n_C})$, then
\begin{align*}
b &=(n - |W| - |A| - (|B|-B_1) - |C|,B_1,\ldots, B_{n_B}), \quad \text{and} \\
c &=(m - |W| - |A| - |B| - (|C|-C_1), C_1,\ldots, C_{n_C}).
\end{align*}
\end{prop}


\begin{cor}
Let $S$ be a linear relation in $\mathbb C^{n}$ and assume that its Weyr characteristic is $(W,A,B,C)$. Then
\begin{align*}
n=|W|+|A|+|B|+|C|
\end{align*}
is the smallest possible number of rows for matrices $E,F\in\C^{n\times m}$ such that the range representation $FE^{-1}$ of the matrix pencil $P(s)=sE-F$ coincides with $S$. In this case, the Weyr characteristic $(w,a,b,c)$ of the pencil $P(s)$ is given by
\begin{eqnarray*}
w=W, \quad a=A,
\end{eqnarray*}
and if $B=( B_1,\ldots, B_{n_B})$ and $C=( C_1,\ldots, C_{n_C})$, then
\begin{align*}
b &=(B_1,B_1,B_2,\ldots, B_{n_B}), \quad \text{and} \\
c &=(m - |W| - |A| - |B| - (|C|-C_1), C_1,C_2,\ldots, C_{n_C}).
\end{align*}

\end{cor}

\section{Rank one perturbations of matrix pencils via linear relations}\label{sec:pert}

 There is a lot of recent work on low-rank perturbations of matrix pencils, see 
 e.g.\ \cite{Batz14,DD07,DMT08,DD16} for generic perturbation results and 
 \cite{BaraDodi20,BaraRoca19,BaraRoca20,DodiStos20,GT17} for non-generic 
 perturbation results. 
 The  above  mentioned  non-generic  results  are  
 mainly  based  on  the  use  
 of invariant factors and the concept of majorization of finite sequences. 
 
 Here we follow a different approach.
 Our goal is to obtain perturbation results for matrix pencils via the Weyr characteristic of their kernel 
 or range representations. 
 Due to the results developed in the Sections \ref{sec:weyr_kernel} and \ref{sec:weyr_range} these Weyr characteristics can be deduced
 from the Weyr characteristic of the underlying, unperturbed pencil.
Recent results for perturbations of linear 
 relations \cite{LMPPTW18} describe
 the change of the  Weyr characteristic under perturbations, and, again due to 
 Sections \ref{sec:weyr_kernel} and \ref{sec:weyr_range},
 allow to give bounds for the change of the Weyr characteristic of the perturbed matrix pencil.
 The choice whether to use kernel or range
 representation depends on the appearance of the rank-one pencil. 

Along this section we consider two matrix pencils,
\begin{align*}
P(s):= sE-F \quad \text{and} \quad Q(s):= sG-H,
\end{align*}
where $E,F,G,H\in\C^{n\times m}$, and assume that they are a rank one perturbation of each other, i.e.\
\begin{align*}
\rank(P-Q)=1.
\end{align*}
Every rank one pencil in $\C^{n\times m}$ can be written in one of the following ways:
\begin{eqnarray}\label{adin}
& (su-v)w^* &\mbox{ or}
\end{eqnarray}
\begin{eqnarray}\label{dwa}
&  w(su^*-v^*),
\end{eqnarray}
where  $(u,v)\in\C^{m}\times\C^m$, $(u,v)\neq (0,0)$ and  $w\in\C^n$, $w\neq 0$, see \cite{GT17}.

The following definition, taken from \cite{ABJT}, specifies the idea of a rank-one perturbation for linear relations.

\begin{defi}\label{laplata16}
Let $S$ and $T$ be linear relations in a vector space $X$. Then $S$ is called a {\em rank one perturbation} of $T$ (and vice versa) if
\begin{equation*}
\max\left\{\dim\frac{S}{S\cap T},\,\dim\frac T{S\cap T}\right\} = 1.
\end{equation*}
\end{defi}

Before returning to matrix pencils, we prove a simple result on rank one perturbations of linear relations.

\begin{lemma}\label{utiles}
Let $S$ and $T$ be linear relations in an $m$-dimensional vector space $X$. Then, the following conditions are equivalent:
\begin{itemize}
\item[\rm (a)] $S$ is a rank one perturbation of $T$;
\item[\rm (b)] $S^{-1}$ is a rank one perturbation of $T^{-1}$;
\item[\rm (c)] $S^\bot$ is a rank one perturbation of $T^\bot$.
\end{itemize}
\end{lemma}

\begin{proof}
In order to prove the equivalences, it suffices to show that
\begin{align*}
\dim\left(\frac{S}{S\cap T}\right) &=\dim \left(\frac{S^{-1}}{S^{-1}\cap T^{-1}}\right)=\dim \left(\frac{T^\bot}{S^\bot\cap T^\bot}\right). 
\end{align*}
for two linear relations $S$ and $T$ in $X$. It is easy to see that $(S\cap\ T)^{-1}=S^{-1}\cap T^{-1}$ and also $\dim S^{-1}= \dim S$. Hence,
\begin{align*}
\dim \left(\frac{S^{-1}}{S^{-1}\cap T^{-1}}\right) &=\dim S^{-1} - \dim \big(S^{-1}\cap T^{-1}\big) = \dim S - \dim \big((S\cap T)^{-1}\big) \\ &= \dim S - \dim (S\cap T)= \dim\left(\frac{S}{S\cap T}\right).
\end{align*}
Since $X$ is finite dimensional, we have that 
$$
S^\bot\cap T^\bot=\left\{(x_1+x_2, y_1+y_2) \,:\, (x_1,y_1) \in S, (x_2,y_2)\in T \right\}^\bot,
$$
and $\dim S^\bot\cap T^\bot=2m- \dim S - \dim T + \dim S\cap T$. Hence,
\begin{align*}
\dim \left(\frac{T^\bot}{S^\bot\cap T^\bot}\right) &= \dim T^\bot - \dim S^\bot\cap T^\bot\\
&= \dim T^\bot -(2m - \dim S - \dim T + \dim S\cap T) \\
&= \dim S -  \dim S\cap T =\dim\left(\frac{S}{S\cap T}\right).
\end{align*}
\end{proof}


In the following lemma  we show that rank one perturbations of matrix pencils can always be reformulated as rank one  perturbations of their kernel or of their range representation.

\begin{lemma}\label{Miranda}
Let $P(s)=sE-F$ and $Q(s)=sG-H$ be matrix pencils in $\mathbb C^{n\times m}$.
\begin{itemize}
\item[\rm (a)]
If $(Q-P)(s)$ is a rank one matrix pencil as in \eqref{dwa}, 
then 
$$
FE^{-1} \quad \mbox{and} \quad HG^{-1},
$$
either coincide or they are rank one perturbations of each other.
\item[\rm (b)]
If $(Q-P)(s)$ is a rank one matrix pencil as in \eqref{adin}, 
then 
$$
E^{-1}F \quad \mbox{and} \quad  G^{-1}H,
$$
either coincide or they are rank one perturbations of each other.
\end{itemize}
\end{lemma}

\begin{proof}
We first prove item (a). If $(Q-P)(s)$ is a rank one matrix pencil as in \eqref{dwa}, then $G=E+uw^*$ and $H=F+vw^*$. Hence,
\begin{align*}
HG^{-1}=R\begin{bmatrix}G\\H\end{bmatrix}=R\begin{bmatrix}E+uw^*\\F+vw^*\end{bmatrix},
\end{align*}
i.e. $(x,y)\in HG^{-1}$ if and only if there exists $z\in\C^m$ such that $x=(E+uw^*)z$ and $y=(F+vw^*)z$.

If $w^*z=0$ then $(x,y)\in (HG^{-1})\cap (FE^{-1})$. Therefore,
\begin{align*}
HG^{-1} = (HG^{-1})\cap (FE^{-1})+\Span\{(u_0,v_0)\},
\end{align*}
where $u_0=Ew + (w^*w)u$ and $v_0=Fw + (w^*w)v$.
Thus, the range representations are rank one perturbations of each other, cf.\ Definition \ref{laplata16}.

\smallskip
%
%

Now, if $(Q-P)(s)$ is a rank one matrix pencil as in \eqref{adin}, then
\begin{align*}
G^{-1}H= \left(E+wu^*\right)^{-1}(F+wv^*)=N[F+wv^*,-E-wu^*].
\end{align*}
Thus, the orthogonal complements of $E^{-1}F$ and $G^{-1}H$ are of the form
\begin{align*}
(E^{-1}F)^\bot=R\begin{bmatrix} F^*\\-E^* \end{bmatrix} \quad \text{and} \quad (G^{-1}H)^\bot=R\begin{bmatrix} F^*+vw^*\\-E^*-uw^* \end{bmatrix},
\end{align*}
respectively, i.e.\ they are the range representations of the pencils $\widetilde{P}(s)=sE^*+F^*$ and $\widetilde{Q}(s)=s(F^*+vw^*)+(E^*+uw^*)$, which differ in a rank one perturbation of the form \eqref{dwa}. Applying item (a), they either coincide or they are rank one perturbations of each other. Finally,  by Lemma \ref{utiles} we have that $E^{-1}F$ and $G^{-1}H$ are rank one perturbations of each other.
\end{proof}

The following theorem is the main result of this section. It describes the change of (a part of) the Weyr characteristic  of matrix pencils under rank-one perturbations. Very similar results have been previously obtained for regular pencils in \cite[Proposition~4.1]{GT17}, see also \cite{DMT08,GernTrun21};
for singular pencils such a result was given in \cite[Section~7]{LMPPTW18} but only for perturbations of the form \eqref{adin}. 
\begin{theorem}
Given $E, F\in\C^{n\times m}$, assume that $\lambda \in \mathbb C$ is an eigenvalue of the matrix pencil $P(s)=sE-F$ and
denote by $w(\lambda) =(w_k(\lambda))$, $a=(a_k)$ and $b=(b_k)$
the finite sequences corresponding to the Weyr characteristic $(w,a,b,c)$ of the pencil $P$. 
Assume also that $Q(s)$ is another matrix pencil in $\mathbb C^{n\times m}$ such that $\rank(P-Q)=1$. Denote by $\widetilde w(\lambda)=(\widetilde w_k(\lambda))$, $\widetilde a=(\widetilde a_k)$  and $\widetilde b=(\widetilde b_k)$ the corresponding parts of the Weyr characteristic $(\widetilde w, \widetilde a, \widetilde b, \widetilde c)$ of the pencil $Q$. 
\begin{itemize}
\item[\rm (a)] If $(Q-P)(s)=w(su^*-v^*)$ as in \eqref{adin} then, for each $k\geq1$, 
\begin{align*}
|(w_k(\lambda)+b_k)-(\widetilde w_k(\lambda)+\widetilde b_k)|\leq k,\\
|(a_k+b_k)-(\widetilde a_k+\widetilde b_k)|\leq k.
\end{align*}
\item[\rm (b)] If $(Q-P)(s)=(su-v)w^*$ as in \eqref{dwa} then, for each $k\geq 1$, 
\begin{align*}
|(w_k(\lambda)+b_{k+1})-(\widetilde w_k(\lambda)+\widetilde b_{k+1})|\leq k,\\
|(a_k+b_{k+1})-(\widetilde a_k+\widetilde b_{k+1})|\leq k.
\end{align*}
\item[\rm (c)] If $P$ and $Q$ are regular then, for each $k\geq 1$, $b_k=\widetilde b_k=0$  and also 
\begin{align*}
|w_k(\lambda)-\widetilde w_k(\lambda)|\leq 1 \quad \text{and}\quad |a_k-\widetilde a_k|\leq 1,
\end{align*}
independently of the form of the rank one pencil $Q-P$.
\end{itemize}
\end{theorem}

\begin{proof}
First we prove (a). Assume that $S:=E^{-1}F$ and $T:=\left(E+wu^*\right)^{-1}(F+wv^*)$ are the kernel representations  corresponding to the unperturbed $P$ and the perturbed pencil $Q$, respectively. Due to  Lemma \ref{Miranda}, they either coincide or they are rank one perturbations of each other. On the one hand, \cite[Theorem 4.5]{LMPPTW18} and Lemma~\ref{numbers} imply that
$$
\left|\big(W_k(S,\lambda)+B_k(S)\big)-\big( W_k(T,\lambda)+B_k(T)\big)\right|\leq k,
$$
for each $k\geq 1$. Now the first statement in item (a) follows from Theorem~\ref{prop:KCFtoWeyr}, because $W_k(S,\lambda)=w_k(\lambda)$, $ W_k(T,\lambda)=\widetilde w_k(\lambda)$, $B_k(S)=b_k$ and $B_k(T)=\widetilde b_k$.


On the other hand, by Lemma~\ref{numbers with A}, $A_k(S)+B_k(S)= \dim\frac{\mul(S^k)}{\mul(S^{k-1})}$ and $A_k(T)+B_k(T)= \dim\frac{\mul(T^k)}{\mul(T^{k-1})}$. Note that $\mul(S^k)=N((S^{-1})^k)$ and $\mul(T^k)=N((T^{-1})^k)$ for every $k\geq 1$. Also, by Lemma \ref{Miranda}, $S^{-1}$ and $T^{-1}$ are rank one perturbations of each other. Then, \cite[Theorem 4.5]{LMPPTW18} implies that
$$
\left|\big(A_k(S)+B_k(S)\big)-\big(A_k(T)+B_k(T)\big)\right|\leq k,
$$
for each $k\geq 1$ and the second statement in item (a) also follows from Theorem~\ref{prop:KCFtoWeyr}.

\smallskip

Item (b) is proved in a similar way, where Theorem~\ref{prop:KCFtoWeyr2} is used instead of Theorem~\ref{prop:KCFtoWeyr}. Assume that $S:=FE^{-1}$ and $T:=(F+vw^*)\left(E+uw^*\right)^{-1}$ are the range representations corresponding to the unperturbed $P$ and the perturbed pencil $Q$, respectively. In this case $S$ and $T$ are again rank one perturbations of each other, see Lemma \ref{Miranda}.
Then, applying \cite[Theorem 4.5]{LMPPTW18} and Lemma~\ref{numbers} we get
$$
\left|\big(W_k(S,\lambda)+B_k(S)\big)-\big( W_k(T,\lambda)+B_k(T)\big)\right|\leq k,
$$
for each $k\geq 1$. By Theorem~\ref{prop:KCFtoWeyr2}, we have that $W_k(S,\lambda)=w_k(\lambda)$, $ W_k(T,\lambda)=\widetilde w_k(\lambda)$, $B_k(S)=b_{k+1}$ and $B_k(T)=\widetilde b_{k+1}$, and the first inequality in (b) follows.  

Analogously, we can use \cite[Theorem 4.5]{LMPPTW18} to show that
$$
\left|\big(A_k(S)+B_k(S)\big)-\big(A_k(T)+B_k(T)\big)\right|\leq k, \quad k\geq 1,
$$
and the second inequality follows again from Theorem~\ref{prop:KCFtoWeyr2}.

%
\smallskip

Finally, if $P(s)=sE-F$ and $Q(s)=sG-H$ are regular then $\mathcal{R}_c(E^{-1}F)=\mathcal{R}_c(FE^{-1})=\mathcal{R}_c(G^{-1}H)=\mathcal{R}_c(HG^{-1})=\{0\}$ 
and therefore $B_k=\widetilde B_k=0$ for all $k\geq 1$. From \cite[Corollary~4.6]{LMPPTW18} and Lemma~\ref{numbers} we have
\begin{align*}
|W_k(\lambda)-\widetilde W_k(\lambda)|\leq 1 \quad \text{and} \quad |A_k(\lambda)-\widetilde A_k(\lambda)|\leq 1.
\end{align*}
Thus, (c) follows after applying Theorem~\ref{prop:KCFtoWeyr} if $Q-P$ is as in \eqref{adin} and Theorem~\ref{prop:KCFtoWeyr} if $Q-P$ is as in \eqref{dwa}, respectively. 
\end{proof}

\appendix

\section{Invariant characteristics of matrix pencils} \label{sec:kron}

In the following, we recall some basic definitions and results for matrix pencils and use the opportunity to fix some notation. Given a pair of matrices $E, F\in \mathbb{C}^{n\times m}$, the matrix pencil
\begin{equation*}
P(s):= sE-F,\quad s\in \mathbb C,
\end{equation*}
is called \emph{singular} if one of the following conditions holds: $n\neq m$ or $n=m$ and $\det(P(s)) = 0$ for all $s\in\C$. Otherwise the pencil is called \emph{regular}.

The \emph{rank {\rm (or} normal rank{\rm )}} of the pencil $P$ is the largest possible size of a regular submatrix pencil of $P(s)$. We denote it by $\rank(P)$.

\begin{defi}[\cite{S}]
A complex number $\lambda$ is a \emph{finite eigenvalue} of $P(s)=sE-F$, 
if the rank of the matrix $P(\lambda)$ is less than $\rank(P)$. In addition,
$\infty$ is an eigenvalue of $P$ if zero is an eigenvalue of the dual pencil $sF-E$.
\end{defi}

Assume that $P_1$ and $P_2$ are two matrix pencils. They are called {\em strictly equivalent} if there are invertible matrices $U\in\C^{n\times n}$ and $V\in\C^{m\times m}$ such that $P_2(s) = UP_1(s)V$ for all $s\in\C$.

Recall that every pencil $P(s)=sE-F$ is strictly equivalent to its \textit{Kronecker canonical form}, see e.g.~\cite{BergTren12,BergTren13,G59}. To introduce this form the following notation is used: for
$k\in\N$ consider the square matrices
\begin{align*}
N_k:=\left[\begin{array}{cccc}
0&&&\\
1&0&&\\
&\ddots&\ddots&\\
&&1&0
\end{array}
\right]\in\C^{k\times k}.
\end{align*}
For a multi-index  $\alpha=(\alpha_1,\ldots,\alpha_l)\in\N^l$ with absolute value $|\alpha|:=\sum_{i=1}^l\alpha_i$, define the block diagonal matrix
\begin{align}\label{KissMeMore}
N_{\alpha}:=\diag(N_{\alpha_1},\ldots,N_{\alpha_l})\in\C^{|\alpha|\times |\alpha|}.
\end{align}
On the other hand, for $k\in\N$, $k> 1$, consider the rectangular matrices
\begin{align*}
K_k=\begin{smallbmatrix}
&&&\\
1 & 0 &  & \\
 & \ddots & \ddots &  & \\
 && 1 & 0\\
 \end{smallbmatrix},\
L_k=\begin{smallbmatrix}
&&&\\
0 & 1 &  & \\
 & \ddots & \ddots &  & \\
 && 0 & 1\\
 \end{smallbmatrix} \in\C^{(k-1)\times k}.
\end{align*}
For some multi-index ${\bf \alpha}=(\alpha_1,\ldots,\alpha_l)\in\N^l$
with $\alpha_j>1$ for all $j=1, \ldots , l$, we define the block quasidiagonal matrices
\begin{equation*}
\begin{aligned}
K_\alpha:= \diag(K_{\alpha_1},\ldots,K_{\alpha_l}),\
L_\alpha:= \diag(L_{\alpha_1},\ldots,L_{\alpha_l})\in\C^{(|\alpha|-l) \times |\alpha|}.
\end{aligned}
\end{equation*}
One extends the above notion to
multi-indices $\alpha=(\alpha_1,\ldots,\alpha_l)\in\N^l$ where some (but not all) entries are
equal to $1$ by adding a zero column at every position where
an entry in $\alpha$ is equal to one. As an example consider
\begin{align*}
K_{(2,1)}:=  \begin{bmatrix}1 & 0 & 0  \end{bmatrix}
\quad \mbox{and} \quad
K_{(3,2,1,1)} :=
 \begin{bmatrix} 1 & 0 & 0 & 0 & 0&0&0 \\
                 0 & 1 & 0 & 0 & 0&0&0 \\
                 0 & 0 & 0 & 1 & 0&0&0
  \end{bmatrix}.
\end{align*}
The case that all entries in $\alpha$ are equal to one may also occur, for a more
detailed exposition of the corresponding definitions for $K_\alpha$ and $L_\alpha$ see, e.g., \cite{BTW16}. In this case, if $sE_0-F_0$ with  $E_0,F_0\in\C^{n_0\times n_0}$, and $(1,\ldots,1)\in\N^{l}$ then the block diagonal pencil is defined as follows
\begin{equation*}
\begin{bmatrix}
sE_0-F_0&0\\0&sK_{(1,\ldots,1)}-L_{(1,\ldots,1)}\end{bmatrix}:=
\begin{bmatrix}
sE_0-A_0,0_{n_0\times l}
\end{bmatrix}
\end{equation*}
and, correspondingly,
\begin{equation}\label{Derulo4}
\begin{bmatrix}
sE_0-F_0&0\\
0&sK^\top_{(1,\ldots,1)}-L^\top_{(1,\ldots,1)}\end{bmatrix}:=
\begin{bmatrix}
sE_0-A_0\\
0_{l\times n_0}
\end{bmatrix}.
\end{equation}


According to Kronecker \cite{K90}, for a matrix pencil $P(s)=sE-F$ there exist invertible matrices $U\in\C^{m\times m}$ and $V\in\C^{n\times n}$ such that $V(sE-F)U$ has a block quasi-diagonal form
\begin{align}\label{KCF2}
V(sE-F)U=\begin{bmatrix} sI_{n_0}-J_0& 0&0&0 \\0& sN_\alpha-I_{|\alpha|}&0&0\\ 0&0& sK_{\beta}-L_{\beta} &0\\ 0&0&0& sK_{\gamma}^\top-L_{\gamma}^\top\end{bmatrix}
\end{align}
for some $J_0\in\C^{n_0\times n_0}$ in Jordan normal form, which is unique up to a permutation of its Jordan blocks, and multi-indices $\alpha\in\N^{n_\alpha}$, $\beta\in\N^{n_\beta}$, $\gamma\in\N^{n_{\gamma}}$ which are unique up to a permutation of their entries, see~\cite{K90} and also \cite[Chapter XII]{G59} or \cite{KM06}.
For our purposes it is important to sort the entries of the multi-indices and the Jordan blocks in  $J_0$ non-increasingly.

The entries in the Segre characteristic  $s(J_0;\lambda)$ of an eigenvalue $\lambda$ of $J_0$ are called the \emph{degrees of the finite elementary divisors} associated to the matrix pencil $P(s)=sE-F$, see e.g. \cite{BergTren12, G59}. The entries of the multi-index $\alpha$ are called the \emph{degrees of the infinite elementary divisors}.
The \emph{column minimal indices} are defined by $\varepsilon_i=\beta_i-1$ for $i=1,\ldots, n_\beta$, and $\eta_j=\gamma_j-1$ ($j=1,\ldots, n_\gamma$) are called the \emph{row minimal indices}, cf.\ \cite{BergTren12, G59}.

\bibliographystyle{amsplain}

\begin{thebibliography}{123}
\bibitem{Arens}
{\sc R.~Arens},
{\em Operational calculus of linear relations},
Pacific J.\ Math.\ 11 (1961), 9--23.

\bibitem{ABJT}
{\sc T.Ya.\ Azizov,  J.\ Behrndt, P.\ Jonas, and C.\ Trunk},
{\em Compact and finite rank perturbations of linear relations in Hilbert spaces},
Integral Equations Operator Theory 63 (2009), 151--163.

\bibitem{BaraDodi20}
{\sc I.\ Baragaña, M.\ Dodig, A.\ Roca, and M.\ Stošić}, {\em Bounded rank perturbations of regular pencils over arbitrary fields}, Linear Algebra Appl.\ 601 (2020), 180--188.

\bibitem{BaraRoca19}
{\sc I.\ Baraga\~na and A.\ Roca}, \emph{Weierstrass structure and eigenvalue placement of regular matrix
pencils under low rank perturbations}, SIAM J. Matrix Anal. Appl. 40 (2019), pp. 440--453.

\bibitem{BaraRoca20}
{\sc I.\ Baraga\~na and A.\ Roca}, {\em Rank-one perturbations of matrix pencils},
Linear Algebra Appl.\ 606 (2020), 170--191.


\bibitem{Batz14}
{\sc L.\ Batzke}, {\em Generic rank-one perturbations of structured regular matrix pencils}, Linear Algebra
Appl. 458 (2014), 638--670.



\bibitem{BennByer01}
{\sc P.~Benner and R.~Byers},
{\em Evaluating products of matrix pencils and collapsing matrix products},
Numer.\ Linear Algebra Appl.\ 8 (2001), 357--380.

\bibitem{BB06}
{\sc P.~Benner and R.~Byers},
{\em An arithmetic for matrix pencils: theory and new algorithms},
 Numer.\ Math.\  103 (2006), 539--573.

\bibitem{BSTW20}
{\sc T.~Berger, H.~de Snoo, C.~Trunk, and H.~Winkler},
{\em Linear relations and their singular chains}, 
to appear in: Methods Funct.\ Anal.\ Topol.

\bibitem{BSTW21}
{\sc T.~Berger, H.~de Snoo, C.~Trunk, and H.~Winkler},
{\em On a Jordan-like decomposition for linear relations in finite dimensional spaces}, submitted.

\bibitem{Gap}
{\sc T.~Berger, H.~Gernandt, C.~Trunk,  H.~Winkler, and M. Wojtylak},
{\em The gap distance to the set of singular matrix pencils},
Linear Algebra Appl.\ 564 (2019), 28--57.

 \bibitem{BergTren12}
{\sc T.~Berger and S.~Trenn},
{\em The quasi-{K}ronecker form for matrix pencils},
SIAM J.\ Matrix Anal.\  Appl.\ 33 (2012), 336--368.

\bibitem{BergTren13}
{\sc T.~Berger and S.~Trenn},
{\em Addition to ``{T}he quasi-{K}ronecker form for matrix pencils''},
SIAM J.\ Matrix Anal.\  Appl.\ 34 (2013), 94--101.

\bibitem{BTW16} {\sc T.\ Berger, C.\ Trunk, and H.\ Winkler},
{\em Linear relations and the {K}ronecker canonical form}, Linear Algebra Appl. 488 (2016), 13--44.

\bibitem{ChenResp21} {\sc Y.\ Chen and W.\ Respondek},
{\em Geometric Analysis of Differential-Algebraic Equations via Linear Control Theory}, SIAM J.\ Control Optim. 59 (2021), 103--130.


\bibitem{c}
{\sc R.\ Cross},
Multivalued Linear Operators,
Monographs and Textbooks in Pure and Applied Mathematics 213,
Marcel Dekker, Inc., New York, 1998.

\bibitem{DD07} {\sc F.\ De Ter\'{a}n and F.\ Dopico},
{\em Low rank perturbation of Kronecker structures without full rank}, SIAM J.\ Matrix Anal.\ Appl.\ 29 (2007), 496--529.

\bibitem{DMT08} {\sc F.\ De Ter{\'a}n, F.\ Dopico and J.\ Moro},
{\em Low rank perturbation of {W}eierstrass structure}, SIAM J.\ Matrix Anal.\ Appl.\ 30 (2008),  538--547.

\bibitem{DD16}
{\sc F. De Ter\'an and F. Dopico}, {\em Generic change of the partial multiplicities of regular matrix
pencils under low-rank perturbations}, SIAM J. Matrix Anal. Appl.\ 37 (2016), 823--835.

\bibitem{DS2}
{\sc A.\ Dijksma and H.S.V.\ de Snoo},
{\em Symmetric and selfadjoint
relations in Krein spaces~II.} Ann.\ Acad.\ Sci.\ Fenn.\ Math.\ 12
(1987), 199--216.

\bibitem{DodiStos20}
{\sc M.\ Dodig and M.\ Sto\v{s}i\'c}, {\em Rank One Perturbations of Matrix Pencils}, SIAM J.\ Matrix Anal.\ Appl.\ 41, 1889-1911.


\bibitem{G59}
{\sc F.\ Gantmacher}, Theory of Matrices,
  Chelsea, New York, 1959.



\bibitem{Wafa} 
{\sc H.\ Gernandt, N.\ Moalla, F.\  Philipp, W.\ Selmi, C.\ Trunk}, {\em  Invariance of the Essential Spectra of Operator Pencils}. Operator Theory: Advances and Applications 278. Birkhäuser, Cham, 2020. 

\bibitem{GT17} {\sc H.\ Gernandt and C.\ Trunk},
{\em Eigenvalue placement for regular matrix pencils under rank one perturbations}, SIAM J.\ Matrix Anal.\ Appl.\ 38 (2017),  134--154.

\bibitem{GernTrun21} {\sc H.\ Gernandt and C.\ Trunk},
{\em The spectrum and the Weyr characteristics of operator pencils and linear relations}, Syphax Journal of Mathematics: Nonlinear Analysis, Operator and Systems 1 (2021), 73--89.



\bibitem{KK86}
{\sc N.~Karcanis and G.~Kalogeropoulos}, {\em On the Segré, Weyr characteristic of right (left) regular matrix pencils}, Int.\ J.\ of Control 44.4 (1986), 991-1015.

\bibitem{K90}
{\sc L.\ Kronecker},
Algebraische Reduction der Schaaren bilinearer Formen, Sitzungsber.\ Akad.\ Berlin (1890), 1225–-1237.

\bibitem{KM06}
{\sc P.\ Kunkel and V.\ Mehrmann},
Differential-Algebraic Equations. Analysis and Numerical Solution,
EMS Publishing House, Z\"{u}rich 2006.

\bibitem{LMPPTW18} {\sc L.\ Leben, F.\ Martinez Peria, F.\ Philipp, C.\ Trunk and H.\ Winkler}, {\em Finite rank perturbations of linear relations and singular matrix pencils}, 
Complex Anal.\ Oper.\ Theory 15 (2021) 37.

\bibitem{LS09} {\sc R.\ Lippert and G.\ Strang},
{\em The Jordan forms of AB and BA}, Electron.\ J.\ Linear Algebra 18 (2009), 281--288.

\bibitem{M33} {\sc C.C.\ MacDuffee},
The Theory of Matrices, Springer, Berlin, 1933.

\bibitem{SandDeSn05}
{\sc A.~Sandovici, H.S.V.~de~Snoo, and H.~Winkler},
{\em The structure of linear relations in {E}uclidean spaces},
Linear Algebra Appl.\ 397 (2005), 141--169.

\bibitem{SandDeSn07}
{\sc A.~Sandovici, H.S.V.~de~Snoo, and H.~Winkler}
{\em Ascent, descent, nullity, defect, and related notions
for linear relations in linear spaces},
Linear Algebra Appl.\ 423 (2007), 456--497.


\bibitem{S15} {\sc H.\ Shapiro},
Linear Algebra and Matrices. Topics for a Second Course, Providence, RI: American Mathematical Society (AMS), 2015.

\bibitem{S99}
{\sc H.\ Shapiro},
{\em The Weyr Characteristic},
American Mathematical Monthly 106 (1999), 919-929.


\bibitem{S}
{\sc J.G.\ Sun},
{\em Orthogonal projections and the perturbation of the eigenvalues of singular pencils},
J.\ Comput.\ Math.\ 1 (1983), 63-–74.
\end{thebibliography}

\end{document}